\documentclass[aap,noinfoline]{imsart}
\pdfoutput=1
\usepackage[top=1.5in,bottom=1.5in,left=1.5in,right=1.5in]{geometry}
\setattribute{author}{prefix}{\empty}
\setattribute{copyright}{text}{\empty}
\usepackage[usenames,dvipsnames]{xcolor}
\definecolor{RefColor}{rgb}{0,0,.65}
\RequirePackage[colorlinks,linkcolor=RefColor,citecolor=RefColor,urlcolor=RefColor]{hyperref}
\hypersetup{pdfinfo={Subject={ }}}
\RequirePackage[OT1]{fontenc}
\usepackage{amsmath,amssymb,amscd,amsfonts,amsthm,mathtools}
\usepackage[numbers]{natbib}

\usepackage[capitalize]{cleveref}

\crefname{appsec}{Appendix}{Appendices}

\usepackage{thmtools}
\usepackage{booktabs}
\usepackage{tikz}
\usetikzlibrary{calc}
\usepackage{dsfont}
\usepackage{centernot}
\usepackage{appendix}
\usepackage{graphicx}

\DeclareGraphicsExtensions{.pdf,.png,.jpg}
\graphicspath{ {include/} }

\startlocaldefs
\endlocaldefs

\definecolor{CommentColor}{rgb}{0,0,.50}
\newcounter{margincounter}

\declaretheoremstyle[
  postheadhook = {\hspace*{\parindent}},
  postfoothook={\vspace{.7em}},
  postheadspace=-.5em,
  mdframed={
    backgroundcolor=gray!10!white, 
    hidealllines=true, 
    innertopmargin=5pt, 
    innerbottommargin=7pt, 
    skipabove=8pt,
    skipbelow=15pt,
    nobreak=true}
]{grayboxed} 
\declaretheorem[style=plain,qed=$\triangleleft$]{auxtheorem}
\declaretheorem[style=plain,sibling=auxtheorem,qed=$\triangleleft$]{theorem}
\declaretheorem[style=plain,sibling=auxtheorem,qed=$\triangleleft$]{lemma}
\declaretheorem[style=plain,sibling=auxtheorem,qed=$\triangleleft$]{corollary}
\declaretheorem[style=plain,sibling=auxtheorem,qed=$\triangleleft$]{proposition}
\declaretheorem[{numbered=no},style=grayboxed,name={${(\alpha,t)}$-urn}]{urnscheme}
\declaretheorem[{numbered=no},style=grayboxed,name={${(\alpha,t)}$-graph}]{graphscheme}

\declaretheorem[style=definition,qed=$\triangleleft$,sibling=theorem]{remark}

\begin{document}

\begin{abstract}
  We study preferential attachment mechanisms in random graphs that are parameterized by (i) a constant bias affecting the degree-biased distribution on the vertex set and (ii) the distribution of times at which new vertices are created by the model. The class of random graphs so defined admits a representation theorem reminiscent of residual allocation, or ``stick-breaking'' schemes. We characterize how the vertex arrival times affect the asymptotic degree distribution, and relate the latter to neutral-to-the-left processes. Our random graphs generate edges ``one end at a time'', which sets up a one-to-one correspondence between random graphs and random partitions of natural numbers; via this map, our representation induces a result on (not necessarily exchangeable) random partitions that generalizes a theorem of Griffiths and Span\'o. A number of examples clarify how the class intersects with several known random graph models.
\end{abstract}

\begin{frontmatter}
  \title{Preferential attachment and vertex arrival times}
    \author{\fnms{Benjamin }\snm{Bloem-Reddy}\ead[label=e1]{benjamin.bloem-reddy@stats.ox.ac.uk}}
    \and
    \author{\fnms{Peter\ }\snm{Orbanz}\corref{}\ead[label=e2]{porbanz@stat.columbia.edu}}
    \affiliation{University of Oxford and Columbia University}
    \address{Department of Statistics\\
      24--29 St. Giles'\\
      Oxford OX1 3LB, UK\\
      \printead{e1}
    }    
    \address{Department of Statistics\\
      1255 Amsterdam Avenue\\
      New York, NY 10027, USA\\
      \printead{e2}
    }    
\end{frontmatter}

\def\kword#1{\textbf{#1}}
\def\xspace{\mathbf{X}}
\def\borel{\mathcal{B}}
\def\mean{\mathbb{E}}
\def\condind{{\perp\!\!\!\perp}}
\def\ie{i.e.\ }
\def\eg{e.g.\ }
\def\equdist{\stackrel{\text{\rm\tiny d}}{=}}
\def\equas{=_{\text{\rm\tiny a.s.}}}
\def\braces#1{{\lbrace #1 \rbrace}}
\def\bigbraces#1{{\bigl\lbrace #1 \bigr\rbrace}}
\def\Bigbraces#1{{\Bigl\lbrace #1 \Bigr\rbrace}}
\def\simiid{\sim_{\mbox{\tiny iid}}}
\def\Law{\mathcal{L}}
\def\iid{i.i.d.\ }
\def\ind#1{\text{\tiny #1}}
\newcommand{\argdot}{{\,\vcenter{\hbox{\tiny$\bullet$}}\,}}
\renewcommand\labelitemi{\raisebox{0.35ex}{\tiny$\bullet$}}
\def\Teven{\mathcal{T}_{\text{\tiny\rm 2}}}
\def\map{\Phi}
\def\DB{\text{\rm DB}}
\def\G{\mathbb{G}}
\def\T{\mathbb{T}}
\def\P{\mathbb{P}}
\def\R{\mathbb{R}}
\def\Barabasi{Barab\'asi}
\def\bbE{\mathbb{E}}
\def\indicator{\mathds{1}}

\def\Bernoulli{\text{\rm Bernoulli}}
\def\PoissonPlus{\text{\rm Poisson}_{+}}
\def\Poisson{\text{\rm Poisson}}
\def\Uniform{\text{\rm Uniform}}
\def\GammaDist{\text{\rm Gamma}}
\def\BetaDist{\text{\rm Beta}}
\def\NBPlus{\text{\rm NB}_{+}}
\def\NB{\text{\rm NB}}
\def\Geom{\text{\rm Geom}}
\def\CRP{\text{\rm CRP}}
\def\EGP{\text{\rm EGP}}
\def\MittagLeffler{\text{\rm ML}}
\def\GGdist{\text{\rm GGa}}

\def\gvar{\mathcal{G}}
\def\bvar{\mathcal{B}}
\def\mlvar{\mathcal{M}}

\def\egc{e.g.}

\def\ct{D}
\def\randperm{\tilde{\sigma}}
\def\randPerm{\tilde{\Sigma}}
\def\ordct{\ct^{\downarrow}}
\def\iatime{\delta}
\def\iaTime{\Delta}
\def\iaInf{\iatime_{1:\infty}}
\def\IaInf{\iaTime_{1:\infty}}
\def\atime{t}
\def\aTime{T}
\def\atimeInf{\atime_{1:\infty}}
\def\aTimeInf{\aTime_{1:\infty}}
\def\aDist{\Lambda}
\def\aDistiid{\tau}
\def\on{\mid}

\def\bbN{\mathbb{N}}
\def\bbP{\mathbb{P}}

\def\sparsity{\varepsilon}
\newcommand{\widesim}[1][1.5]{
  \mathrel{\scalebox{#1}[1]{$\sim$}}
}
\newcommand{\limscale}[2]{\overset{\scriptscriptstyle{#1 \uparrow #2}}{\widesim[1.25]}}

The term \emph{preferential attachment} describes generative mechanisms for random graph models
that select the terminal vertices of a new edge with probability biased by the vertex degrees.
These models come in many shapes and guises
\citep[\egc][]{Barabasi:Albert:1999,Berger:etal:2014,Hofstad:2016,Pekoz:Rollin:Ross:2017}, and are 
often motivated by their ability to generate (and hence explain) power law distributions. 
Degree-biased selection is a form of size bias \citep{Arratia:Goldstein:Kochman:2013:1},
and this interplay between size-biasing and power laws is not confined to random graph models, but
also encountered in random partitions, which are used in population genetics,
machine learning, and other fields \citep[\egc][]{Pitman:2006,deBlasi:etal:2015,Broderick:Jordan:Pitman:2012}. 
In partition models, power laws arise as heavy-tailed distributions of block sizes.
Size-biased sampling as such, however, need not result in a power laws:
The most basic form of 
size-biased sampling from a countable number of categories is a type of P\'olya urn with 
an unbounded number of colors, or, equivalently, the one-parameter Chinese
restaurant process \citep{Pitman:2006}. It does not generate a power law. 
To obtain power laws, plain size-biased sampling can be modified in two ways:
\begin{enumerate}
\renewcommand\labelenumi{(\roman{enumi})}
\item By biasing the size-biased probability of each category downward. 
\item By forcing new categories to arrive at a faster rate than that induced by plain size-biased sampling.
\end{enumerate}
An example is the two-parameter Chinese restaurant process with parameter $(\alpha,\theta)$,
which modifies the Chinese restaurant process with a single parameter $\theta$---a model that corresponds
to plain size-biased sampling---by effectively (i) damping the size bias by a constant offset $\alpha$, and (ii) increasing the rate at which new categories arrive. 
An example of (ii) is the \Barabasi--Albert random graph model, in which vertices
arrive at fixed, constant time intervals; if these times were instead determined at random by size-biased sampling, intervals would grow over time.

The premise of this work is to study preferential attachment mechanisms in random graph models
by explicitly controlling these two effects:
\begin{enumerate}
  \renewcommand\labelenumi{(\roman{enumi})}
\item The attachment probability, proportional to the degree $\deg_k$ of each vertex $k$, is biased
  by a constant offset as ${\deg_k-\,\alpha}$.
\item Vertex arrival times are taken into account, by explicitly conditioning the generation process on a (random or non-random) sequence of given times.
\end{enumerate}
The result is a class of random graphs parametrized by the offset $\alpha$ and a sequence $t$ of vertex
arrival times. Each such $(\alpha,t)$-graph can be generalized by randomizing the arrival times,
\ie to an $(\alpha,T)$-graph for a random sequence $T$. 
Preferential attachment models that constantly bias the attachment probability 
have been thoroughly studied \citep{Mori:2005,Hofstad:2016}. We consider the range ${\alpha\in(-\infty,1)}$,
and the case ${\alpha\in[0,1)}$ turns out to be of particular interest.
The effects (i) and (ii) are not independent, and in models with a suitable exchangeability 
property, the effect of $\alpha$ can equivalently be induced by controlling the law of the arrival times.
In this sense, (ii) can provide more control over the model than (i).

\cref{sec:representation} characterizes $(\alpha,T)$-graphs by a representation in terms of independent beta random variables, reminiscent of stick-breaking
constructions of random partitions.

\cref{sec:graphs:urns} considers implications for random partitions and urns. 
${(\alpha,T)}$-graphs generate edges ``one end at a time'', updating the vertex degrees after each step. Although such a scheme differs from the usual preferential attachment model, it is similar to so-called ``sequential'' versions considered by \cite{Berger:etal:2014,Pekoz:Rollin:Ross:2017}. This sets up a one-to-one correspondence between multigraphs and partitions of natural numbers:
There is a bijection $\Phi$ such that
\begin{equation*}
  \Phi(\text{partition})=\text{graph}\;,
\end{equation*}
which translates our results on graphs into statements about partitions.
If $G$ is an $(\alpha,T)$-graph, the random partition $\Phi^{-1}(G)$ may or may not be 
exchangeable. The subclass of such partitions that are exchangeable are precisely 
the exchangeable Gibbs partitions \citep{Gnedin:Pitman:2006,Pitman:2006}. Arrival times in such partitions,
known as \emph{record indices}, have been studied by \citet{Griffiths:Spano:2007}. 
Broadly speaking, our results recover those of
Griffiths and Span\`o if $\Phi^{-1}(G)$ is exchangeable, but 
show there 
is a larger class of random partitions---either of Gibbs type, or not exchangeable---for
which similar results hold. Non-exchangeable examples include partitions defined
by the Yule--Simon process \citep{Yule:1925,Simon:1955}. Additionally, 
our representation result for graphs yields an analogous representation for this class
of partitions; it also relates the 
work of \citet{Griffiths:Spano:2007} to that of \citet*{Berger:etal:2014}
on Benjamini--Schramm limits of certain random preferential attachment graphs.

\cref{sec:degree:asymptotics} studies degree asymptotics of $(\alpha,T)$-graphs.
Properly scaled degree sequences of such graphs converge. The limiting degrees 
are neutral-to-the-left sequences of random variables that satisfy a number of distributional identities.
We characterize cases in which power laws emerge, and 
relate the behavior of the degree distribution to sparsity. 
The range of power laws achievable is constrained by whether or not the average degree is bounded.

\cref{sec:examples} discusses examples, and shows how the class of $(\alpha,T)$-graphs overlaps 
with several known models, such as
the \Barabasi--Albert model \citep{Barabasi:Albert:1999}, edge exchangeable graphs \citep{Crane:Dempsey:2016,Cai:etal:2016,Janson:2017aa}, 
and the preferential attachment model of \citet*{Aiello:Chung:Lu:2001,Aiello:Chung:Lu:2002}.
We obtain new results for some of these submodels. For preferential attachment graphs, for example, 
limiting degree sequences are known to satisfy various
distributional identities \cite{Pekoz:Rollin:Ross:2017,James:2015aa,Janson:2006}. These results assume
fixed intervals between vertex arrival times. 
We show similar results hold if arrival times are random. 
Perhaps most closely related is the work of \citet*{Pekoz:Rollin:Ross:2017aa}
on random immigration times in a two-color P\'{o}lya urn, which corresponds to a certain
$(\alpha,T)$-model with \iid interarrival times. 
We use this correspondence to answer a question posed in \citep{Pekoz:Rollin:Ross:2017aa} about
interarrival times with geometric distribution.

\section{Preferential attachment and arrival times}
\label{sec:representation}

Consider an undirected multigraph $g$, possibly with self-loops, with a countably infinite number of edges. The graph models considered
in the following insert edges into the graph one at a time. It is hence convenient to represent $g$ as
a sequence of edges
\begin{equation}
  \label{eq:graph:1}
  g=\bigl( (l_1,l_2),(l_3,l_4),\ldots \bigr)
  \qquad\text{ where }l_j\in\mathbb{N}\text{ for all }j\in\mathbb{N}\;.
\end{equation}
Each pair ${(l_{2n-1},l_{2n})}$ represents an undirected edge connecting the vertices $l_{2n-1}$ and $l_{2n}$. 
The vertex set of $g$ is ${\mathbf{V}(g):=\braces{l^*_1,l^*_2,\ldots}}$, the set of all distinct values
occurring in the sequence. We assume vertices are enumerated in order of appearance, to wit
\begin{equation}
  \label{eq:graph:2}
  l_1=1 \qquad\text{ and }\qquad l_{j+1}\leq \max\braces{l_1,\ldots,l_j}+1\quad\text{ for all }j\in\mathbb{N}\;.
\end{equation}
Consequently, $\mathbf{V}(g)$ is either a consecutive finite set $\braces{1,\ldots,m}$, or the entire set
$\mathbb{N}$. Let $\G$ be the set of multigraphs so defined, equipped with the
topology inherited from the product space $\mathbb{N}^{\infty}$, which makes it a standard
Borel space. For our purposes, a \kword{random graph} is a random element 
\begin{equation*}
  G=\bigl((L_1,L_2),(L_3,L_4),\ldots)
\end{equation*}
of $\G$, for $\mathbb{N}$-valued random variables $L_n$. Note that the same setup can be used to model directed multigraphs.

For a graph $g$, let ${g_n:=(l_j)_{j\leq 2n}}$ denote the subgraph given by the first
$n$ edges, and $\deg_k(n)$ the degree of vertex $k$ in $g_n$.
The \kword{arrival time} of vertex $k$ is
\begin{equation*}
  t_k:=\min\braces{j\in\mathbb{N}\,\vert\,l_j=k}\;,
\end{equation*}
with ${t_k=\infty}$ if $g$ has fewer than $k$ vertices. The set of possible arrival time
sequences is
${\T:=\braces{(1=t_1<t_2<\ldots\leq\infty)}}$.
Arrival times in ${\Teven:=\braces{t\in\T\,\vert\,t_k\text{ even for }k>1}}$
is a sufficient, though not necessary, condition for $g$ to be a connected graph; it is necessary and sufficient for each $g_n$ to be connected. 
If ${T=(T_1,T_2,\ldots)}$ is a random sequence of arrival times, the
interarrival times
\begin{equation*}
  \Delta_k:=T_k-T_{k-1} \qquad\text{ where }T_0:=0
\end{equation*}
are random variables with values in ${\mathbb{N}\cup\braces{\infty}}$.

\subsection{Degree-biased random graphs}
The term preferential attachment describes a degree bias: Recall that the 
\kword{degree-biased} distribution on the vertices of a graph 
$g_n$ with $n$ edges is ${P(k\,;g_n):=\deg_k(n)/2n}$. 
We embed $P$ into a one-parameter family of
laws
\begin{equation*}
  P_{\alpha}(k\,;g_n)
  :=
  \frac{\deg_k(n)-\alpha}{2n-\alpha|\mathbf{V}(g_n)|}
  \qquad
  \text{ for }\alpha\in(-\infty,1)\;,
\end{equation*}
the \kword{$\alpha$-degree biased} distributions.
Both $P$ and $P_{\alpha}$ are defined on a graph $g_n$, in which each edge
is either completely present or completely absent. To permit edges to be 
generated ``one end at a time'', we observe $P_{\alpha}$ can be rewritten as
\begin{equation} \label{eq:p:alpha}
  P_\alpha(k\,;l_1,\ldots,l_j)=
  \frac{|\braces{i\leq j|l_j=k}|-\alpha}{j-\alpha\max\braces{l_1,\ldots,l_j}}
  \qquad\text{ for }k\leq \max\braces{l_1,\ldots,l_j}\;,
\end{equation}
which is well-defined even if $j$ is odd. 
\begin{graphscheme} Given are ${\alpha\in(-\infty,1)}$ and ${t\in\T}$. Generate
  ${L_1,L_2,\ldots}$ as
  \vspace{-.4em}
  \begin{equation*}
    L_n:=k \quad\text{ if } n=T_k
    \qquad\text{ and }\qquad
    L_n\sim P_{\alpha}(\argdot\,;L_1,\ldots,L_{n-1})\quad\text{ otherwise}\;.
  \end{equation*}
\end{graphscheme}
Then ${G:=((L_1,L_2),(L_3,L_4),\ldots)}$ is a random graph, whose law we denote ${\DB(\alpha,t)}$. 
The sequence $t$ may additionally be randomized: We call $G$
an $(\alpha,T)$-graph if its law is $\DB(\alpha,T)$, for some
random element $T$ of $\T$. Examples of multigraphs generated using different distributions for $T$ are shown in \cref{fig:examples}. 

As a consequence of the family of laws \eqref{eq:p:alpha}, the finite-dimensional distributions of $(\alpha,t)$-graphs have a simple product form.
\begin{proposition} \label{prop:exact:probability}
  Let $G_{n}$ be an $(\alpha,t)$-graph with $n/2$ edges and $k$ vertices. Then
  \begin{align} \label{eq:exact:probability}
    \bbP_{\alpha,t}[G_{n/2} & = ((L_1,L_2),\dotsc,(L_{n-1},L_{n}))] \nonumber \\
    & = \frac{1}{\Gamma(n-k\alpha)} \prod_{j=1}^{k} \frac{\Gamma(T_j - j\alpha) \Gamma(\#\{L_i=k \mid i \leq n\} - \alpha)}{\Gamma(T_j - 1 - (j-1)\alpha + \delta_1(j)) \Gamma(1 - \alpha)} \;.
  \end{align}
\end{proposition}

\subsection{Representation result}
Fix ${\alpha\in(-\infty,1)}$ and ${t\in\T}$.
Let ${\Psi_1,\Psi_2,\ldots}$ be independent random variables with $\Psi_1 = 1$,
\begin{equation}
  \label{eq:sb:1}
  \Psi_j\sim\text{Beta}\bigl(1-\alpha,t_j-1-(j-1)\alpha\bigr) \quad \text{for} \quad j\geq 2 \;,
\end{equation}
and define
\begin{equation}
  \label{eq:sb:2}
  W_{j,k}:=
  \sum_{i = 1}^{j}
  \Psi_i\prod_{\ell = i+1}^k (1-\Psi_{\ell})
  \quad\text{ and  }\quad
  I_{j,k}:=[W_{j-1,k},W_{j,k}) \text{ with }W_{0,k}=0\;.
\end{equation}
Note that $W_{j,k}=\prod_{\ell=j+1}^k (1-\Psi_{\ell})$ and $W_{k,k}=1$. Hence, ${\cup_{j=1}^k I_{j,k}=[0,1)}$. Generate a random sequence ${U_1,U_2,\ldots\simiid\Uniform[0,1)}$. For each $n$, let $t_{k(n)}$ be the preceding arrival time, 
\ie the largest $t_k$ with ${t_k\leq n}$, and set
\begin{equation}
  \label{eq:sb:3}
  L_n :=\begin{cases}
  k(n) & \text{ if } n=t_{k(n)}\\
  j \text{ such that } U_n\in I_{j,k(n)} & \text{ otherwise}
  \end{cases}\;.
\end{equation}
Then ${H(\alpha,t):=((L_1,L_2),\ldots)}$ is a random element of $\G$.
\begin{theorem}
  \label{theorem:sb}
  A random graph $G$ is an $(\alpha,T)$-graph 
  if and only if
  ${G\equdist H(\alpha,T)}$ for some ${\alpha\in(-\infty,1)}$
  and a random element $T$ of $\T$.
\end{theorem}

Products of the form \eqref{eq:sb:2}, for the same sequence of beta variables $\Psi_j$, previously have
appeared in two separate contexts:
\citet{Griffiths:Spano:2007} identify $\Psi_j W_{j,\infty}$ as the limiting relative block sizes in 
exchangeable Gibbs partitions, conditioned on the block arrival times. This corresponds to the special
case of \cref{theorem:sb} where the random variables ${(L_1,L_2,\ldots)}$ define an exchangeable Gibbs
partition (see \cref{sec:graphs:urns,sec:exchangeable:partitions}). In work of \citet*{Berger:etal:2014}, a version of \eqref{eq:sb:1}--\eqref{eq:sb:3} arises
as the representation of the Benjamini--Schramm limit of certain preferential attachment graphs
(in which case all interarrival times are fixed to a single constant).

That two problems so distinct lead to the same 
(and arguably not entirely obvious) distribution 
raises the question whether \eqref{eq:sb:1} can be understood in a more
conceptual way. One such way is by regarding the graph as a recursive sequence of P\'{o}lya urns: Conditionally on an edge attaching to one of the first $k$ vertices, it attaches to vertex $k$ with probability $\Psi_k$ and one of the first $k-1$ with probability $1-\Psi_k$, and so on for $k-1,\dotsc,2$. A related interpretation is in terms of the special properties of beta and gamma random variables. 
Let $\gvar_a$ and $\bvar_{a,b}$ generically denote a gamma random variable with parameters $(a,1)$ and a beta
variable with parameters $(a,b)$. 
Beta and gamma random variables satisfy a set of relationships sometimes referred to collectively 
as the \emph{beta-gamma algebra} \citep[e.g.][]{Revuz:Yor:1999}. These relationships revolve around the fact that, if $\gvar_a$ and
$\gvar_b$ are independent, then 
\begin{equation}
  \label{eq:beta:gamma:algebra}
  \bigl(\gvar_{a+b},\bvar_{a,b}\bigr)
  \equdist 
  \Bigl(\gvar_a+\gvar_b,\frac{\gvar_a}{\gvar_a+\gvar_b}\Bigr)\;,
\end{equation}
where the pair on the left is independent, and so is the pair on the right.
In the context of $(\alpha,T)$-graphs, conditionally on the sequence ${\Delta_1,\Delta_2,\ldots}$ of interarrival times, 
generate two sequences of gamma variables
\begin{equation*}
  \gvar^{(1)},\gvar^{(2)},\ldots\simiid\GammaDist(1-\alpha,1)
  \qquad\text{ and }\qquad
  \gvar_{\Delta_2-1},\gvar_{\Delta_3-1},\ldots\;,
\end{equation*}
all mutually independent given ${(\Delta_k)}$. 
The variables $\Psi_k$ can then be represented as
\begin{equation}
  \label{eq:recursion}
  \Psi_j\quad\equdist\quad\frac{\gvar^{(j)}}{\sum_{i\leq j}\gvar^{(i)}+\sum_{i<j}\gvar_{\Delta_{i+1}-1}}\quad=:\quad\Psi_j'\;.
\end{equation}
Such recursive fractions are not generally independent, but as a consequence of \eqref{eq:beta:gamma:algebra},
equality in law holds even jointly, ${(\Psi_1,\Psi_2,\ldots)\equdist(\Psi_1',\Psi_2',\ldots)}$,
recovering the variables in \cref{theorem:sb}.
Identity \eqref{eq:beta:gamma:algebra} further implies
${\bvar_{a,b+c}\equdist \bvar_{a,b}\bvar_{a+b,c}}$, again with independence on the right.
Abbreviate 
\begin{equation*}
  \tau_j:=t_j-1+\alpha(j-1) \qquad\text{ such that \eqref{eq:sb:1} becomes }\qquad \Psi_j\equdist \bvar_{1-\alpha,\tau_j}\;.
\end{equation*}
The recursion \eqref{eq:recursion} then implies
\begin{equation*}
  \bvar_{1-\alpha,\tau_j}\equdist \bvar_{1-\alpha+\tau_{j-1},\Delta_j-\alpha}\bvar_{1-\alpha,\tau_{j-1}}
  \quad\text{ hence }\quad
  \Psi_j|\Psi_{j-1}\equdist\Psi_{j-1}\bvar_{1-\alpha+\tau_{j-1},\Delta_j-\alpha}\;,
\end{equation*}
with independence on the right of both identities.
Informally, one may think of ${\gvar^{(k)}}$ as an (unnormalized) propensity of vertex $k$ to attract edges, of those edges attaching to one of the first $k$ vertices. The requisite
normalization in \eqref{eq:recursion} 
depends on propensities of previously created vertices
(represented by the variables ${\gvar^{(1)},\ldots,\gvar^{(k-1)}}$), and contributions of the ``head start'' given to previously created vertices 
(represented by the variables $\gvar_{\Delta_j - 1}$).

\section{Graphs and urns}
\label{sec:graphs:urns}

Any graph in $\G$ defines a partition of $\mathbb{N}$, and vice versa.
This fact is used below to classify $\alpha$-degree biased graphs according to the
properties of the random partition they define. More precisely, 
a \kword{partition} of $\mathbb{N}$ is a sequence 
${\pi=(b_1,b_2,\ldots)\subset\mathbb{N}}$ of subsets, called \kword{blocks},
such that each ${n\in\mathbb{N}}$ belongs to one and only one block.
The set of all partitions is denoted ${\mathcal{P}(\mathbb{N})}$, and
inherits the topology of ${\mathbb{N}^{\infty}}$.
A partition can equivalently
be represented as a sequence ${\pi=(l_1,l_2,\ldots)}$ of block labels, where
${l_j=k}$ means ${j\in b_k}$. There is hence a bijective map
\begin{equation*}
  \map:\mathcal{P}(\mathbb{N})\rightarrow\G
  \qquad\text{ given by }\qquad
  (l_1,l_2,\ldots)\mapsto\bigl((l_1,l_2),(l_3,l_4),\ldots\bigr)\;,
\end{equation*}
which is a homeomorphism of $\mathcal{P}(\mathbb{N})$ and $\G$.
It identifies blocks of $\pi$ with vertices of ${g=\map(\pi)}$. 
In population genetics, the smallest element of the $k$th block of a partition $\pi$ is known
as a \emph{record index} \cite{Griffiths:Spano:2007}. 
Thus, the $k$th arrival time in $g$ is precisely the $k$th
record index of $\pi$.

The generative process of a random partition $\Pi$ can be thought of as an urn: Start with an empty urn, and add consecutively
numbered balls one at a time, each colored with a randomly chosen color. Colors may reoccur, and are
enumerated in order of first appearance. Let $B_k(n)$ be the set of all balls sharing
the $k$th color after $n$ balls have been added. For ${n\rightarrow\infty}$, one obtains
a random partition ${\Pi=(B_1,B_2,\ldots)}$ of $\mathbb{N}$, with blocks ${B_k:=\cup_n B_k(n)}$.
In analogy to the $(\alpha,t)$-graphs above, we define:
\begin{urnscheme} Given are ${\alpha\in(-\infty,1)}$ and ${t\in\T}$.
  \vspace{-.4em}
\begin{itemize}
\item If ${n=t_k}$ for some $k$, add a single ball of a new, distinct color to the urn.
\item Otherwise, add a ball of a color already present in the urn, where
  the $j$th color is chosen with probability proportional to
  ${|B_j(n)|-\alpha}$. 
\end{itemize}
\end{urnscheme}
A familiar special case of such an urn is the P\'olya urn with $m$ colors, obtained for
${\alpha=0}$ and ${t=(1,2,\ldots,m,\infty,\infty,\ldots)}$. Another is the two-parameter
Chinese restaurant process \citep{Pitman:2006}, also known as the Blackwell--MacQueen urn \citep{Blackwell:MacQueen:1973,Pitman:1996aa}: If $t$ is randomized by
generating ${(T_1=1,T_2,T_3,\ldots)}$ according to
\begin{equation} \label{eq:crp:arrivals}
  \mathbb{P}[T_{k+1}=T_{k}+t \on T_k]
  =
  (\theta + \alpha k) 
  \frac{ \Gamma(\theta + T_k) \Gamma(T_k + t - 1 - \alpha k) }{ \Gamma(\theta + T_k + t) \Gamma(T_k - \alpha k) }\;,
\end{equation}
for some ${\theta>-\alpha}$, 
the partition has law ${\text{CRP}(\alpha,\theta)}$.

In general, $(\alpha,t)$-urns define a class of random partitions $\Pi$ that are \emph{coherent}, in the sense that 
\begin{align*}
  \bbP( \Pi_{n-1} = \{B_1,\dotsc,B_k\} ) = \sum_{j=1}^{k+1} \bbP( \Pi_{n} = \mathcal{A}_{n\to j}(\Pi_{n-1}) ) \;,
\end{align*}
where $\mathcal{A}_{n\to j}(\Pi_{n-1})$ denotes the operation of appending $n$ to block $B_j$ in $\Pi_{n-1}$. Partitions for which these probabilities depend only on the sizes of the blocks, and which are therefore invariant under permutations of the elements, are exchangeable random partitions \cite{Pitman:2006}. That is, there is an \emph{exchangeable partition probability function} (EPPF) $p(\cdot)$, symmetric in its arguments, such that
\begin{align*}
  p(|B_1|,\dotsc,|B_k|) = \bbP(\Pi_n = \{ B_1,\dotsc, B_k \}) \;,
\end{align*}
which is invariant under the natural action of the symmetric group. A special subclass are the exchangeable partitions of \emph{Gibbs type}, for which the EPPF has the unique product form \cite{Gnedin:Pitman:2006}
\begin{align} \label{eq:egp:eppf}
  p( |B_1|,\dotsc,|B_k| ) = 
    V_{n,k} \prod_{j=1}^k \frac{\Gamma(|B_j| - \alpha)}{\Gamma(1 - \alpha)} \;,
\end{align}
for a suitable sequence of coefficients $V_{n,k}$ satisfying the recursion
\begin{align} \label{eq:egp:recursion}
  V_{n,k} = (n - \alpha k) V_{n+1,k} + V_{n+1,k+1} \;.
\end{align}
The distribution of the arrival times can be deduced from \eqref{eq:egp:eppf} and \eqref{eq:egp:recursion} as
\begin{align} \label{eq:egp:arrivals}
  \bbP[T_{k+1} = T_k + t \on T_k] = \frac{\Gamma(T_{k} + t - 1 - \alpha k)}{\Gamma(T_k - \alpha k)} \frac{V_{T_{k} + t,k+1}}{V_{T_k,k}} \;,
\end{align}
of which \eqref{eq:crp:arrivals} for the CRP is a special case. Denote the law of $T_1,\dotsc,T_k$ generated by \eqref{eq:egp:arrivals} as $P_{\alpha,V}(T_1,\dots,T_k)$.

Alternatively, consider the $(\alpha,T)$-urn counterpart of the EPPF, given in \cref{prop:exact:probability},
\begin{align} \label{eq:cppf}
  p_{\alpha,T}( |B_1|,\dotsc,|B_k|; & \ T_1,\dots,T_k ) = \bbP[\Pi_n = \{B_1,\dotsc,B_k\} \mid T_1,\dots,T_k ]  \\
  & = \frac{1}{\Gamma(n - k\alpha)} \prod_{j=1}^{k} \frac{\Gamma(T_j - j\alpha) \Gamma(|B_j| - \alpha)}{\Gamma(T_j - 1 - (j-1)\alpha + \delta_1(j)) \Gamma(1 - \alpha)} \nonumber \;.
\end{align}
Define
\begin{align} \label{eq:cond:gibbs:v}
  V_{n,k}^{\alpha,T} :=  \frac{1}{\Gamma(n - k\alpha)} \prod_{j=1}^{k} \frac{\Gamma(T_j - j\alpha)}{\Gamma(T_j - 1 - (j-1)\alpha + \delta_1(j))} \;,
\end{align}
in which case \eqref{eq:cppf} takes on the Gibbs-like form
\begin{align*}
  p_{\alpha,T}( |B_1|,\dotsc,|B_k|; & \ T_1,\dots,T_k ) = V_{n,k}^{\alpha,T} \prod_{j=1}^k \frac{\Gamma(|B_j| - \alpha)}{\Gamma(1 - \alpha)}  \;.
\end{align*}
This general formula holds for all $(\alpha,T)$-urns. In the case that $\Pi$ is exchangeable, these relationships imply a further characterization of exchangeable Gibbs partitions: \eqref{eq:egp:eppf} is obtained by marginalizing the arrival times from \eqref{eq:cppf} according to $P_{\alpha,V}$.
\begin{proposition}
  Let $\Pi$ be a random partition generated by an $(\alpha,T)$-urn, with finite-dimensional conditional distributions given by \eqref{eq:cppf}. Then $\Pi$ is exchangeable if and only if there exists some sequence of coefficients $V = (V_{n,k})$ satisfying
  \begin{align*}
    V_{n,k} & = \sum_{\substack{T_1,\dotsc,T_k \\ T_k \leq n}} \frac{P_{\alpha,V}(T_1,\dots,T_k) }{\Gamma(n - k\alpha)} \prod_{j=1}^{k} \frac{\Gamma(T_j - j\alpha)}{\Gamma(T_j - 1 - (j-1)\alpha + \delta_1(j))} 
     = \bbE[V_{n,k}^{\alpha,T}] \;,
  \end{align*}
  for all $k\leq n$, in which case \eqref{eq:egp:eppf} holds and $\Pi$ is an exchangeable Gibbs partition.
\end{proposition}

It is straightforward to verify that $\Phi(\Pi)$ is an $(\alpha,t)$-graph if and only if
$\Pi$ is an $(\alpha,t)$-urn. This correspondence is used in \cref{sec:examples} to 
classify some ${(\alpha,t)}$-graphs according to the urns they define. 
It also allows us to translate properties of random graphs into properties of random partitions, 
and vice versa. \cref{theorem:sb} implies the following result on partitions, 
which gives a representation of exchangeable Gibbs partitions.
\begin{corollary}
  \label{corollary:G:S}
  A random partition $\Pi$ is an ${(\alpha,t)}$-urn if and only if
  it is distributed as ${\Pi\equdist (L_1,L_2,\ldots)}$,
  for variables $L_n$ generated according to  \eqref{eq:sb:1}--\eqref{eq:sb:3}.
\end{corollary}

\section{Degree asymptotics}
\label{sec:degree:asymptotics}

Let $G$ be an $(\alpha,t)$-graph, and $G_n$ the subgraph given by its first $n$ edges. The 
\kword{degree sequence} of $G_n$ is the vector $\mathbf{D}(n)=(\deg_k(n))_{k\geq 1}$, where vertices are ordered by appearance. Denote by $m_d(n)$ the number of vertices in $G_n$ with degree $d$. The \kword{empirical degree distribution}
  \begin{align*}
    (p_d(n))_{d\geq 1} := |\mathbf{V}(g_n)|^{-1}(m_d(n))_{d\geq 1}
  \end{align*}
is the probability that a vertex sampled uniformly at random from $G_n$ has degree $d$. The degree sequence and the degree distribution as $G_n$ grows large are characterized by the scaling behavior induced by $\alpha$ and $t$, which yields  
power laws and related properties.

\subsection{Linear and sub-linear regimes}
\label{sec:linear:sublinear}

As will become clear in the next section, the scaling behavior of $(\alpha,t)$-graphs is the result of products of the form
\begin{align} \label{eq:tail:product}
  W_{j,k} = \prod_{i=j+1}^k (1-\Psi_i) \quad \text{as} \quad k\to\infty \;,
\end{align}
where $(\Psi_j)_{j>1}$ are as in \eqref{eq:sb:1}. In particular, two regimes of distinct limiting behavior emerge. To which of the two regimes an $(\alpha,t)$-graph belongs is determined by whether or not $W_{j,k}$ converges to a non-zero value as $k\to\infty$. 

We consider $(\alpha,t)$-graphs that satisfy the following assumption:
\begin{align} \label{eq:vertex:arrival:rate}
  |\mathbf{V}(G_n)|/n^{\sigma} \xrightarrow[n\to\infty]{\text{\small a.s.}} \mu_{\sigma}^{-\sigma} \quad \text{ for some } \quad 0 < \sigma \leq 1 \quad \text{ and } \quad 0 < \mu_{\sigma} < \infty \;.
\end{align}
Slower vertex arrival rates (\egc, logarithmic) result in graphs that are almost surely dense (see \cref{sec:degree:distributions}), and as such exhibit less interesting structural properties. For example, 
in order to generate power law distributions in $(\alpha,t)$-graphs, the asymptotic arrival rate must be super-logarithmic, which follows from work on exchangeable random partitions and can be read from \cite[Chapter 3]{Pitman:2006}.

For a growing graph sequence satisfying the assumption \eqref{eq:vertex:arrival:rate}, consider the limiting average degree,
\begin{align*}
  \lim_{n\to\infty} \bar{d}_n = \lim_{n\to\infty} \frac{2n}{\mu^{-\sigma}_{\sigma}n^{\sigma}} = \lim_{n\to\infty} \frac{2}{\mu^{-\sigma}_{\sigma}} n^{1-\sigma} \;.
\end{align*}
The average degree is almost surely bounded if $\sigma=1$, which we call the \kword{linear} regime;  for $\sigma\in(0,1)$, the \kword{sub-linear} regime, it diverges. This is a consequence of \cref{prop:tail:product:convergence} below: For a graph $G_n$ on $k(n)$ vertices, the probability that the end of edge $n+1$ is attached to vertex $j$ is equal to $\Psi_j W_{j,k(n)}$, which results in vertex $j$ participating in a constant proportion of edges if and only if $ W_{j,k(n)}$ is bounded away from zero as $n$ grows large. 
\begin{proposition} \label{prop:tail:product:convergence}
  For fixed $\alpha\in(-\infty,1)$ and $t\in\T$ such that \eqref{eq:vertex:arrival:rate} is satisfied for some $\sigma\in(0,1]$, let $W_{j,k}$ be as in \eqref{eq:tail:product}. Then for each $j\geq 1$, $W_{j,k}$ converges almost surely as $k\to\infty$ to some random variable $W_{j,\infty}$, which is non-zero if and only if $\sigma<1$.
\end{proposition}
\begin{remark}
  For slower vertex arrivals (\egc logarithmic) or when the limiting number of vertices is finite, $W_{j,k(n)}$ also converges to a non-zero value.
\end{remark}

\subsection{Limiting joint distributions of degree sequences for given arrival times}
\label{sec:deg:dist:fixed}

The previous section suggests that the limit of the scaled degrees should depend on the random variables $(\Psi_j)_{j>1}$. Indeed, for any $(\alpha,t)$-graph $G$, it can be shown that for any $r\in\bbN_+$, 
\begin{align} \label{eq:limit:seq:full}
 \bigl( n^{-1} \deg_j(n) \bigr)_{1\leq j \leq r} \xrightarrow[n\to\infty]{\text{\small a.s.}} \bigl( \xi_j  \bigr)_{1\leq j \leq r} \quad \text{where} \quad \xi_j \equdist \Psi_j \prod_{i=j+1}^{\infty} (1-\Psi_i) \;,
\end{align}
and $(\Psi_j)_{j>1}$ are as in \eqref{eq:sb:1}. \Citet{Griffiths:Spano:2007} showed that relative degrees with such a limit uniquely characterize exchangeable Gibbs partitions among all exchangeable partitions; if the random partition ${\Phi^{-1}(G)}$ is exchangeable, that result applies to $G$ (see \cref{sec:exchangeable:partitions}). For a general ${(\alpha,t)}$-graph, ${\Phi^{-1}(G)}$ need not be exchangeable, and indeed there are examples for which ${n^{-1}\mathbf{D}(n)}$ converges to zero, in which case $W_{j,k(n)}$ does, as well. 
In such cases, one may ask more generally whether a finite, non-zero limit
\begin{align*}
  \mathbf{D}_{\infty} := \lim_{n\to\infty} n^{-1/\gamma} \mathbf{D}(n) \;,
\end{align*}
exists for an appropriate scaling exponent $\gamma$. \cref{thm:limiting:degree:sequence} establishes that this is true for $(\alpha,t)$-graphs.
\begin{theorem} \label{thm:limiting:degree:sequence}
  Let $G$ be an $(\alpha,t)$-graph for some $\alpha\in(-\infty,1)$ and $t\in \T$. Then \eqref{eq:limit:seq:full} holds. If $t$ is such that \eqref{eq:vertex:arrival:rate} holds with $\sigma=1$, assume 
  \begin{align*}
    \lim_{j\to \infty} \frac{t_j}{j} = \mu \in (1,\infty) \;.
  \end{align*}
  Then for every $r\in \bbN_+$, there exists a positive, finite constant $M_r(t)$, and positive random variables $\zeta_1,\dotsc,\zeta_r$ such that
  \begin{align*}
    M_r(t) n^{-r/\gamma} \deg_1(n)\dotsm \deg_r(n) \xrightarrow[n\to\infty]{\text{\small a.s.}} \zeta_1 \dotsm \zeta_r  \quad \text{where} \quad \gamma = \frac{\mu - \alpha}{\mu - 1}  \;.
  \end{align*}
  The mixed moments also converge: For any $p_1,\dotsc,p_r > -(1-\alpha)/2$ with $\bar{p}=\sum_{j=1}^r p_j$, there exists some $M_{\bar{p}}(t)\in(0,\infty)$ such that
  \begin{align} \label{eq:limiting:degree:sequence}
    M_{\bar{p}}(t) \mathbb{E}\bigl[ \lim_{n\to\infty} n^{-\bar{p}/\gamma} \deg_1(n)^{p_1}\dotsm \deg_r(n)^{p_r}\bigr] = \mathbb{E}\bigl[\zeta_1^{p_1}\dotsm \zeta_r^{p_r}\bigr]\;. 
  \end{align}
  Furthermore,
  \begin{align} \label{eq:degree:sequence:g}
    \bigl(\zeta_j \bigr)_{1\leq j \leq r} \equdist \bigl( \Psi_j \prod_{i=j+1}^r (1 - \Psi_i)  \bigr)_{1\leq j \leq r} \;,
  \end{align}
  where $\Psi_1=1$ and $(\Psi_j)_{j>1}$ are as in \eqref{eq:sb:1}.
\end{theorem}
In the sub-linear regime, \eqref{eq:limit:seq:full} agrees with and generalizes the result of \cite{Griffiths:Spano:2007} for exchangeable Gibbs partitions (though the proof uses different methods). In the linear regime, the mixed moments of the scaled degrees also converge to those of products of independent beta random variables.
However, the result does \emph{not} completely describe the joint distributions, due to the presence of 
the unknown scaling terms $M_{\bar{p}}(t)$. These terms depend on the moments $p_1,\dotsc,p_r,$ and on $t$, and express the randomness remaining, for large $k(n)$, in $W_{j,k(n)}$ after the part that scales with $n$ is removed; in particular, they result from early fluctuations of the process. \cref{sec:examples} provides stronger results in several cases for which these terms are 
well-behaved.

\subsection{Neutrality}

It was noted in \cref{sec:graphs:urns} that the map $\Phi^{-1}$ from graphs to partitions translates results on graphs into results on partitions. Conversely, one can transfer properties from partitions to graphs. A sequence ${(X_1,X_2,\ldots)}$ of random variables is \kword{neutral-to-the-left} (NTL) if the relative increments
\begin{align*}
  X_1,\frac{X_2}{X_1+X_2},\ldots,\frac{X_j}{\sum_{i=1}^j X_{i}},\ldots
\end{align*}
are independent random variables in (0,1) \cite{Doksum:1974,Griffiths:Spano:2007}. If $\Pi$ is an exchangeable partition, \citet{Griffiths:Spano:2007} show that the limiting relative block sizes of $\Pi$ are NTL if and only if $\Pi$ is an exchangeable Gibbs partition. If so, the random graph ${\Phi(\Pi)}$ has a limiting degree sequence $\mathbf{D}_{\infty}$ that is NTL. Due to the representation in \cref{theorem:sb}, this property generalizes
beyond the exchangeable case:
\begin{corollary}
  The limiting degree sequence $\mathbf{D}_{\infty}$ of an $(\alpha,t)$-graph is NTL.
\end{corollary}

\subsection{Sparsity and power law degrees}
\label{sec:degree:distributions}

Suppose $G_n$ is the subgraph of an $(\alpha,T)$-graph $G$, given by the first $n$ edges. Since $G_n$ is finite,
one can sample a vertex uniformly at random from its vertex set 
and report its degree $D_n$. One can then ask whether the sequence of random degrees $D_n$ converges in distribution
to some limiting variable $D$. We show in this section that that is indeed the case for $(\alpha,T)$-graphs, under some regularity conditions. We also show how the degree distribution is related to the sparsity, or, equivalently, the edge density, of $(G_n)$.

The sequence ${(G_n)}$ is defined to be $\sparsity$-\kword{dense} if
  \begin{align}
    \label{epsilon:density}
    \underset{n\to\infty}{\lim\sup} \frac{n}{|\mathbf{V}(G_n)|^{\sparsity}} = c_{\sparsity} > 0 \quad \text{for some} \quad 
    \sparsity\geq 1 \;.
  \end{align}
If $\sparsity < 2$, the graph sequence is typically called \emph{sparse}; when $\sparsity \geq 2$, the sequence is \emph{dense}. Note that ${\sparsity > 2}$ is only possible for multigraphs. 
The level of sparsity follows from $\sigma$: Graph models in the linear regime 
correspond to $\sparsity = 1$ \cite{Bollobas:Janson:Riordan:2007,Berger:etal:2014,Aiello:Chung:Lu:2002}; graph models in the sub-linear regime with $\sigma > \frac{1}{2}$ have appeared in the literature \cite{Caron:Fox:2017,Veitch:Roy:2015,Crane:Dempsey:2016,Cai:etal:2016}, with $1 < \sparsity < 2$. See \cref{sec:examples} for examples. 

For functions $a$ and $b$, we use the notation
  \begin{align*}
    a(n)\limscale{n}{\infty} b(n)
    \quad:\Leftrightarrow\quad
    \lim_{n\to\infty} a(n)/b(n) \to 1 \;.
  \end{align*}
The sequence $(G_n)$ has \kword{power law degree distribution} with exponent $\eta > 1$ if
  \begin{align}
    p_d(n) = \frac{m_d(n)}{|\mathbf{V}(G_n)|} \xrightarrow[n\to\infty]{} p_d \limscale{d}{\infty} L(d)d^{-\eta} \quad \text{for all large $d$} \;,
  \end{align}
for some slowly varying function $L(d)$, that is, ${\lim_{x\to\infty} L(rx)/L(x) = 1}$ for all ${r>0}$ \cite{Bingham:1989,Feller:1971}. 

In the sub-linear regime, the degree distribution follows from results due to Pitman and Hansen \cite[Lemma 3.11]{Pitman:2006}, see also \cite{Gnedin:Hansen:Pitman:2007}, on the limiting block sizes of exchangeable random partitions (see \cref{sec:exchangeable:partitions} for more details). In particular, if \eqref{eq:vertex:arrival:rate} is satisfied by an $(\alpha,t)$-graph $G^{\alpha}=\Phi(\Pi^{\alpha})$ for $\sigma=\alpha\in(0,1)$, then there exist an exchangeable random partition $\Pi$ and a positive, finite random variable $S$ such that ${|\mathbf{V}(\Phi(\Pi_{2n}))|/n^{\alpha} \xrightarrow[n\to\infty]{\text{\small a.s.}} S}$, $\Pi=\Pi^{\alpha}$ and $S=\mu_{\alpha}$. The limiting degree distribution is
\begin{equation} 
  \label{eq:degree:distn:sublinear}
  p^{\alpha}_d 
  \quad=\quad 
  \alpha \frac{\Gamma(d - \alpha)}{\Gamma(d+1) \Gamma(1 - \alpha)}
  \quad\limscale{d}{\infty}\quad
  \frac{\alpha}{\Gamma(1-\alpha)} d^{-(1+\alpha)} \;.
\end{equation}

In the linear regime, $\sigma=1$, with limiting mean interarrival time $\mu_1$. We show (see \cref{sec:proof:degree:distributions}) that the resulting limiting degree distribution is a generalization of the classical Yule--Simon distribution (which corresponds to $\alpha=0$) \cite{Yule:1925,Simon:1955,Durrett:2006},
\begin{equation} 
  \label{eq:degree:distn:linear}
  p^{\gamma}_d 
  \quad=\quad 
  \gamma \frac{ \Gamma(d - \alpha) \Gamma(1 - \alpha + \gamma) }{ \Gamma(d + 1 - \alpha + \gamma) \Gamma(1 - \alpha) } 
  \quad
  \limscale{d}{\infty}
  \quad
  \gamma \frac{ \Gamma(1 - \alpha + \gamma) }{ \Gamma(1 - \alpha) }d^{-(1+ \gamma)} \;,
\end{equation}
where ${\gamma := \frac{\mu_{1} - \alpha}{\mu_{1} - 1}}$, as in \eqref{eq:limiting:degree:sequence}.

The tail behavior of the two distributions \eqref{eq:degree:distn:sublinear}, \eqref{eq:degree:distn:linear} partition the range of possible values of the power law exponent, as summarized by the following theorem.

\begin{theorem}
  \label{theorem:degree:distn}
  Let $G$ be a random $(\alpha,T)$-graph for some $\alpha \in (-\infty,1)$ and $T\in\T$. If 
  \begin{align*}
    T_j j^{-1/\sigma} \xrightarrow[j\to\infty]{\text{\small a.s.}} \mu_{\sigma} \quad \text{ for some } \quad \sigma \in (0,1] \quad \text{ and } \quad 1 < \mu_{\sigma} < \infty \;,
  \end{align*}
  then $G$ has $\sparsity$-density with $\sparsity=1/\sigma$. If $\sigma=1$, assume that $\bbE[\Delta_j]=\mu_1$, ${\text{Var}(\Delta_j) < \infty}$ for all $j\in\bbN_+$, and $\lvert\text{Cov}(\Delta_i,\Delta_j)\rvert \leq C^2_{\Delta}\lvert i-j\rvert^{-\ell_{\Delta}}$ for all $i,j > 1$, some $C^2_{\Delta}\geq 0$, and some $\ell_{\Delta}>0$. 
  Then the degree distribution converges asymptotically,
  \begin{align*}
    \frac{m_d(n)}{|\mathbf{V}(G_n)|} \xrightarrow[n\to\infty]{\text{\small p}}
    \begin{cases}
      p^{\alpha}_d  & \text{ if } \sigma = \alpha \in (0,1) \\
      p^{\gamma}_d & \text{ if } \sigma = 1
    \end{cases} \;,
  \end{align*}
  which for large $d$ follows a power law with exponent
  \begin{align*}
    \eta = 
    \begin{cases}
      1 + \alpha & \in (1,2) \quad\;\; \text{ if } \sigma = \alpha \in (0,1) \\
      1 + \gamma & \in (2,\infty) \quad \text{ if } \sigma = 1
    \end{cases} \;.
  \end{align*}
\end{theorem}

The distributions \eqref{eq:degree:distn:sublinear}, \eqref{eq:degree:distn:linear} have the following representation, which is useful for generating realizations from those distributions.
\begin{corollary} \label{corollary:rep:1}
  Let $G$ be a random $(\alpha,T)$-graph for some $\alpha \in (-\infty,1)$ and $T\in\T$ satisfying the conditions of \cref{theorem:degree:distn}. Then the degree $D_n$ of a vertex sampled uniformly at random from $G_n$ converges in distribution to $D'$, where $D'$ is sampled as
  \begin{align*}
    & D' \sim \Geom(\bvar)
    \quad\text{ for }\quad
    \bvar \sim 
    \begin{cases}
      \BetaDist(\alpha,1 - \alpha) & \text{ if }  \sigma = \alpha \in (0,1) \\
      \BetaDist(\gamma,1-\alpha) & \text{ if } \sigma = 1
    \end{cases} \;.
  \end{align*}
\end{corollary}
This representation can be refined further: The proof of \cref{theorem:degree:distn} shows, by extending techniques introduced by \citet*{Berger:etal:2014}, that 
the neighborhood of a random vertex can be coupled to a Poisson point process on the unit interval.
That yields the following representation:
\begin{corollary} \label{corollary:rep:2}
  Let $G$ be a random $(\alpha,T)$-graph for some $\alpha \in (-\infty,1)$ and $T\in\T$ satisfying the conditions of \cref{theorem:degree:distn}. Then the degree $D_n$ of a vertex sampled uniformly at random from $G_n$ converges in distribution to $D'$, where $D'$ is sampled as
  \begin{equation*}
      D' \sim \Poisson\left(\frac{1 - \bvar}{\bvar} \ \gvar_{1-\alpha}  \right)
      \quad\text{ for }\quad
      \bvar \sim 
      \begin{cases}
        \BetaDist(\alpha,1) & \text{ if } \sigma = \alpha \in (0,1) \\
        \BetaDist(\gamma,1) & \text{ if } \sigma = 1
      \end{cases} \;.
  \end{equation*}
\end{corollary}
\begin{remark}
  Based on the fact that $(1-\bvar)/\bvar$ is distributed as a so-called beta prime random variable, additional distributional identities may be deduced. To give one, let $\gvar_1$, $\gvar_{\alpha}$, and $\gvar_{\gamma}$ be independent Gamma random variables. Then one can replace $(1-\bvar)/\bvar$ above by $\gvar_1/\gvar_{\alpha}$ (for ${\sigma=\alpha < 1}$) or $\gvar_1/\gvar_{\gamma}$ (for ${\sigma=1}$).
\end{remark}

\subsection{A note on almost surely connected graphs}
\label{sec:connected:graphs}

Suppose one requires each graph in the evolving sequence $(G_n)$ drawn from an $(\alpha,T)$-graph to be
almost surely connected. That holds if and only if ${T\in\Teven}$, \ie if each arrival time after ${T_1=1}$ is even.
A simple way to generate $T\in\Teven$ is to sample $\Delta_2,\Delta_3,\dotsc$ as in the generation of general $T\in\T$, and to set
\begin{align} \label{eq:teven}
  T_2 = 2\Delta_2, \quad T_k = T_{k-1} + 2\Delta_k \quad\text{for}\quad k > 2 \;.
\end{align}
In the sub-linear regime, doubling the interarrival times 
does not affect the degree asymptotics. 
In the linear regime, the change has noticeable affect. For example, suppose the variables $\Delta_k$ above are drawn 
\iid from some probability distribution on $\bbN_+$ with mean $\mu$. 
Then by \cref{theorem:degree:distn}, the limiting degree distribution has power law exponent $\eta_{\text{\tiny\rm 2}} = 1 + \frac{2\mu - \alpha}{2\mu - 1}$. For ${T\not\in\Teven}$, the upper limit of $\eta$ is $\infty$, no matter the value 
of $\alpha$; for $T\in\Teven$, one has ${\eta_{\text{\tiny\rm 2}} < 3 - \alpha}$. Hence, if $\alpha>0$, then $\eta_{\text{\tiny\rm 2}} \in (2,3)$, implying that the limiting degree distribution has finite mean, but unbounded variance for any $\mu$.

\section{Examples}
\label{sec:examples}

\begin{figure}
 \makebox[\textwidth][c]{
  \resizebox{\textwidth}{!}{
    \begin{tikzpicture}
      \begin{scope}
        \begin{scope}[xshift=-5cm]
          \node (g11) {
            \includegraphics[width=4cm,angle=0]{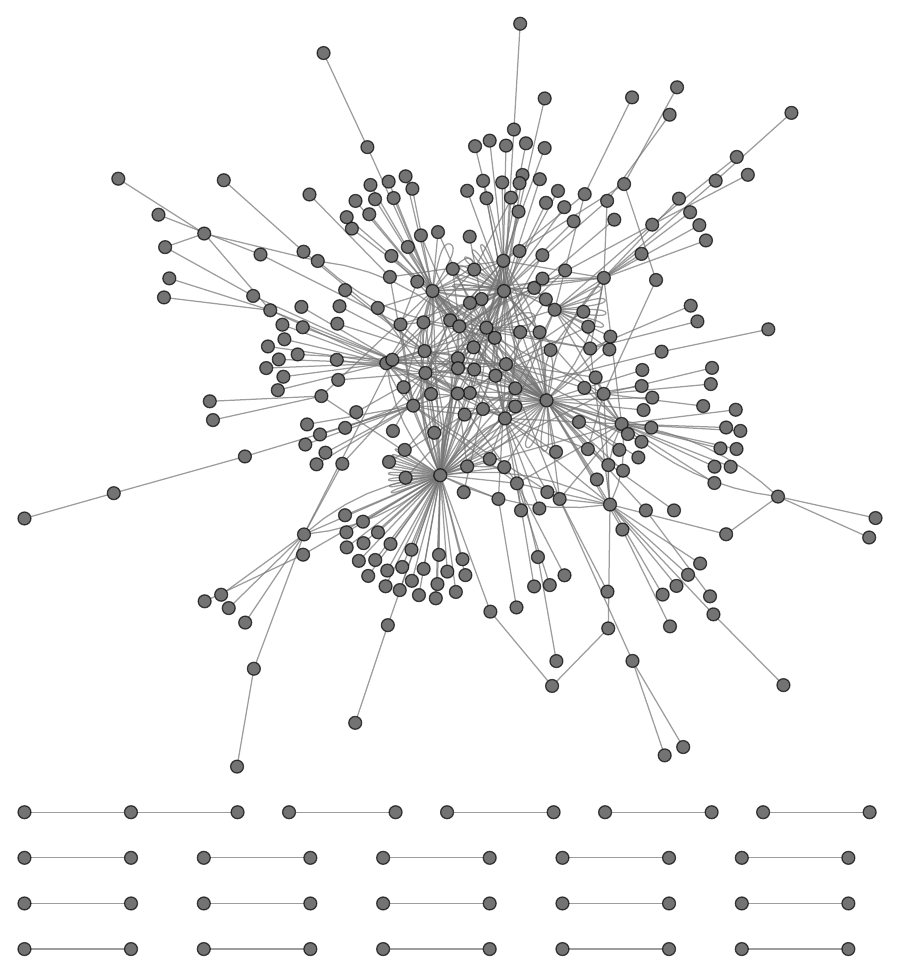}
          };
        \end{scope}
        \begin{scope}[xshift=-0.45cm]
          \node (g12) {
            \includegraphics[width=4.15cm,angle=0]{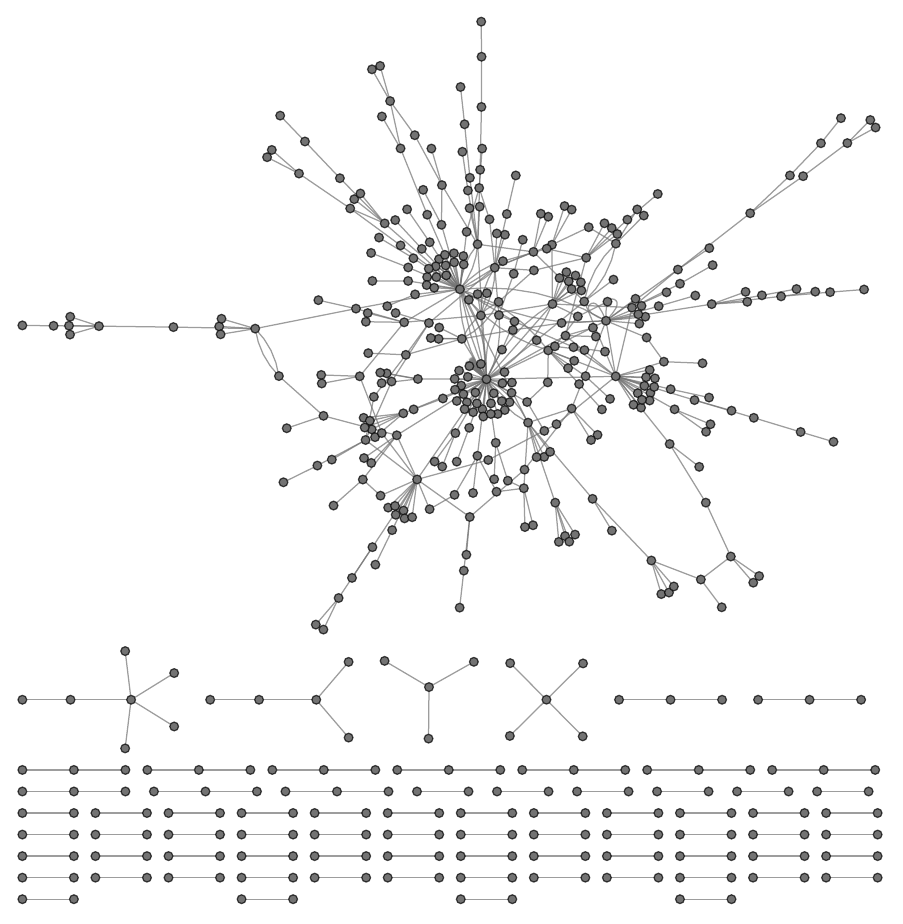}
          };
        \end{scope} 
        \begin{scope}[xshift=4cm]
          \node (g13) {
            \includegraphics[width=4cm,angle=0]{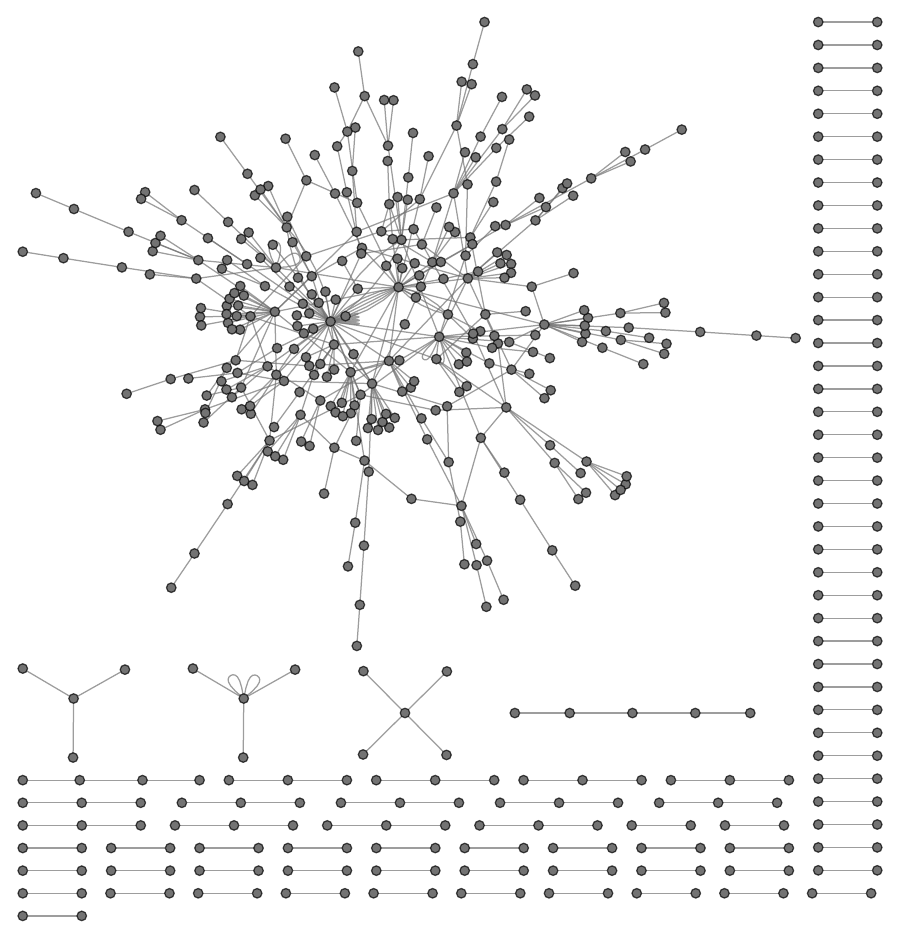}
          };
        \end{scope} 
      \end{scope}
    \end{tikzpicture}
}}
  \caption{Examples of ${(\alpha,T)}$-graphs generated using different arrival time distributions. Each graph has $500$ edges. Left: arrival times generated by $\CRP(\alpha,\theta)$, with $\alpha=0.1$, $\theta=5$. Middle: interarrival times are \iid $\Poisson_+(2)$, with $\alpha=0.1$. Right: interarrival times are \iid $\Geom(0.25)$, with $\alpha= 0.5$.}
  \label{fig:examples}
\end{figure}

We next discuss several subclasses of $(\alpha,T)$-graphs. One is obtained by fixing all interarrival
times to the same, constant value (\cref{sec:preferential:attachment:trees}). This includes the \Barabasi--Albert
random tree as a special case. Other subclasses can be obtained by imposing exchangeability assumptions.
One is that the vertex assignment variables $L_n$ form an exchangeable sequence, and hence that the
induced random partition ${\Phi^{-1}(G)}$ is exchangeable (\cref{sec:exchangeable:partitions}). This subclass
overlaps with the class of ``edge exchangeable'' graphs \citep{Crane:Dempsey:2016,Cai:etal:2016,Janson:2017aa}.
If the interarrival
times ${\Delta_k}$ are exchangeable (\cref{sec:examples:exch:interarrivals}), 
the induced partition is not exchangeable. This case
includes a version of the random graph model of \citep{Aiello:Chung:Lu:2001,Aiello:Chung:Lu:2002,Chung:Lu:2006}.

\subsection{\Barabasi--Albert trees and graphs with constant interarrival time}
\label{sec:preferential:attachment:trees}

  The basic \kword{preferential attachment} model popularized by \cite{Barabasi:Albert:1999}
  generates a random graph as follows: With parameter ${d\in\mathbb{N}_+}$, 
  start with any finite connected graph. 
  At each step, select $d$ vertices in the current graph independently from $P_0$ in \eqref{eq:p:alpha}, add a new
  vertex, and connect it to the $d$ selected vertices (multiple connections are allowed).
  
  The \Barabasi--Albert model can be expressed in terms of a sequence $(L_1,L_2,\ldots)$ with given arrival times
  as follows:
  Start, say, with a graph consisting
  of a single vertex and $d$ self-loops. Thus, ${t_k=1}$ and ${L_1=\ldots=L_{2d}=1}$. Each new
  vertex requires $2d$ stubs, so ${t_{k+1}=t_k+2d}$. At time ${t_{k+1}-1=(2d)k}$, 
  just before the ${(k+1)}$-st vertex arrives, the graph 
  ${G_{kd}=(L_n)_{n\leq 2kd}}$ has $k$ vertices and $kd$ edges.
  For ${t_k\leq n < t_{k+1}}$, we then set
  \begin{equation*}
    L_n=k \quad\text{ if } n \text{ odd }
    \qquad\text{ and }\qquad
    L_n\sim P_0(\argdot;G_{(k-1)d})\quad\text{ if } n \text{ even.}
  \end{equation*}
  The single vertex with self-loops is chosen as a seed graph here only to keep notation
  reasonable. More generally, any graph with $k$ vertices and $n$ edges can be encoded
  in the variables ${L_1,\ldots,L_{2n}}$ and the first $k$ arrival times.

  When $d=1$, the result is a tree (with a loop at the root). When $d\geq 2$, the above sampling scheme does not produce an $(\alpha,t)$-graph. 
  However, the following modified sampling scheme produces an $(\alpha,t)$-graph with $\Delta_j = 2d$ for all $j>1$. Start as before with a graph consisting of a single vertex and $d$ self-loops. When the $k$-th vertex arrives at time $t_k=2(k-1)d+1$, set $L_{t_k}=k$ and for $t_k + 1 \leq n < t_{k+1}$, set
  \begin{equation*}
    L_n\sim P_{\alpha}(\argdot;(L_i)_{i < n}) \;.
  \end{equation*}
  The modified sampling scheme differs from the basic preferential attachment model in that it updates the degrees after each step, allows loops, and does not require that each vertex begin with $d$ edges. Although the connectivity properties of the resulting graph may be substantially different from the \Barabasi--Albert model, the degree properties are similar to modifications that have been considered by \cite{Berger:etal:2014,Pekoz:Rollin:Ross:2017}. In this case, the results of \cref{sec:deg:dist:fixed} can be strengthened, showing that the scaled degrees converge to random variables that satisfy distributional relationships generalizing those of the beta-gamma algebra \eqref{eq:beta:gamma:algebra}, as discussed by \citet{Dufresne:2010}. (See also \cite{Janson:2010}.) These relationships emerge due to the behavior of $W_{j,k}$ as $k\to\infty$, which separates into two pieces: A deterministic scaling factor, and the random variables that appear below.
  \begin{proposition} \label{prop:pa:limits}
    Let $G$ be an $(\alpha,t)$-graph with $\alpha\in(-\infty,1)$ such that for some $d\in\bbN_+$, $t_j = 2d(j-1) + 1$ for all $j\geq 1$. Then for every $r\in\bbN_+$,
    \begin{align*}
      \biggl(\frac{n}{2m}\biggr)^{-\frac{2d-1}{2d-\alpha}} (\deg_j(n))_{1 \leq j \leq r} \xrightarrow[n\to\infty]{\text{\small a.s.}} (\xi_j)_{1 \leq j \leq r} \;.
    \end{align*}
    The vector of random variables $(\xi_j)_{1\leq j \leq r}$ satisfies the following distributional identities: 
    Denote $\bar{\alpha}:= 2d-\alpha$, and define
    \begin{align*}
      Z_r = \prod_{i=1}^{2d-1} \gvar^{1/\bar{\alpha}}_{r+1 - i/\bar{\alpha}} \quad \text{and} \quad Z'_r = \prod_{i=1}^{2d-1} \gvar^{1/\bar{\alpha}}_{1 - i/\bar{\alpha}}  \quad \text{and} \quad Z''_r = Z'_r \prod_{k=1}^{r-1} \bvar_{k\bar{\alpha},1-\alpha} \;,
    \end{align*}
    where all of the random variables are independent of each other and of $(\xi_j)_j$. 
    Then with $\Psi_1=1$, and ${(\Psi_j)_{j>1}\sim\BetaDist(1-\alpha,(j-1)(2d-\alpha))}$, the following distributional identities hold:
    \begin{gather} \label{eq:pa:identity}
      Z_r \cdot \bigl( \xi_j \bigr)_{1\leq j \leq r} 
      \equdist 
      \gvar_{r\bar{\alpha}} \cdot \bigl( \Psi_j \prod_{i=j+1}^r (1 - \Psi_i)  \bigr)_{1\leq j \leq r}
      \\
      Z'_r \cdot \bigl( \xi_j \bigr)_{1\leq j \leq r} 
      \equdist 
      \gvar_{r\bar{\alpha}} \prod_{k=1}^{r} \bvar_{k\bar{\alpha} - 2d + 1,2d-1} \cdot \bigl( \Psi_j \prod_{i=j+1}^r (1 - \Psi_i)  \bigr)_{1\leq j \leq r}
      \\
      Z'_r \cdot \bigl( \xi_j \bigr)_{1\leq j \leq r} 
      \equdist 
      \gvar_{r\bar{\alpha} - 2d + 1} \prod_{k=1}^{r-1} \bvar_{k\bar{\alpha} - 2d + 1,2d-1} \cdot \bigl( \Psi_j \prod_{i=j+1}^r (1 - \Psi_i)  \bigr)_{1\leq j \leq r}
      \\
      Z''_r \cdot \bigl( \xi_j \bigr)_{1\leq j \leq r} 
      \equdist 
      \gvar_{1-\alpha} \cdot \bigl( \Psi_j \prod_{i=j+1}^r (1 - \Psi_i)  \bigr)_{1\leq j \leq r} \;.
    \end{gather}
  \end{proposition}
  \begin{remark}
    For a gamma random variable $\gvar_a$, $\gvar_a^{b}$ has so-called generalized gamma distribution, denoted $\GGdist(a/b,1/b)$. Hence, $Z_r$ above is equal to the product of $\GGdist((r+1)\bar{\alpha}-i,\bar{\alpha})$ random variables, and similarly for $Z'_r$. Generalized gamma random variables also appear in the limits of the preferential attachment models in \cite{Pekoz:Rollin:Ross:2017,Pitman:Racz:2015}, and arise in a range of other applications \cite{Pekoz:Rollin:Ross:2016}.
  \end{remark}

  Results on power law degree distributions in preferential attachment models are numerous \citep[\egc][]{Barabasi:Albert:1999,Aiello:Chung:Lu:2001,Aiello:Chung:Lu:2002,Berger:etal:2014}. It is well-known that the degree distribution of the \Barabasi--Albert tree exhibits a power law with exponent $\eta=3$ \cite{Barabasi:Albert:1999,Bollobas:etal:2001}, which agrees with the following implication of \cref{theorem:degree:distn}.
  \begin{corollary}
    For the constant interarrival time model considered above, the degree distribution converges to \eqref{eq:degree:distn:linear} with $\gamma = (2d-\alpha)/(2d - 1)$. In particular, the $\alpha$-weighted \Barabasi--Albert tree has power law exponent $\eta = 3 - \alpha$.
  \end{corollary}

\subsection{Graphs with exchangeable vertex assignments}
\label{sec:exchangeable:partitions}

Suppose $G$ is a random graph such that the random partition ${\Pi=\Phi^{-1}(G)}$ is exchangeable
(see \cite{Pitman:2006} for more on exchangeable partitions). Equivalently, 
the vertex assignments ${L_1,L_2,\ldots}$ are exchangeable, and there is hence
a random probability measure $\mu$ on $\mathbb{N}$ such that
\begin{equation} \label{eq:cond:iid}
  L_1,L_2,\ldots\mid\mu\simiid\mu\;.
\end{equation}
We first note an implication of \Cref{theorem:degree:distn}. Recall that a random graph is sparse
if its density \eqref{epsilon:density} is ${\sparsity<2}$.
\begin{corollary}
   If a graph generated by an exchangeable partition is sparse
   and has a power law degree distribution, then ${\sigma > 1/2}$, and hence ${\eta\in (3/2,2)}$. 
\end{corollary}
Recall from \cref{sec:graphs:urns} that an exchangeable Gibbs partition is an exchangeable random partition $\Pi$ such that the probability of any finite restriction $\Pi_n$ can be written as
\begin{align*} 
  \mathbb{P}(\Pi_n = \{ \deg_1(n),\dots,\deg_k(n) \}) = 
    V_{n,k} \prod_{j=1}^k \frac{\Gamma(\deg_j(n) - \alpha)}{\Gamma(1 - \alpha)} \;,
\end{align*}
for a suitable sequence of weights $V_{n,k}$. 
\citet{Griffiths:Spano:2007} studied the block arrival times (called, in that context, the record indices) of exchangeable Gibbs partitions.
For the random graph induced by such partitions, their results show that 
\begin{align} \label{eq:degree:proportions}
  \frac{1}{n} \deg_j(n) \xrightarrow[n\to\infty]{\text{\small a.s.}} P_j \equdist \Psi_j \prod_{i = j+1}^{\infty} (1 - \Psi_i) \;,
\end{align}
where $\Psi_j$ is distributed as in \eqref{eq:sb:1}. (This result is also contained in \cref{thm:limiting:degree:sequence}.) They prove that an exchangeable random partition is of Gibbs form if and only if the sequence $(P_j)_{j\geq 1}$ is NTL conditioned on $(T_j)_{j\geq 1}$; this result has implications for some recent network models.

\citet{Crane:Dempsey:2016} and \citet*{Cai:etal:2016} 
call a random graph ${((L_1,L_2),\ldots)}$ \emph{edge exchangeable} 
if there is some random probability measure $\nu$ on ${\mathbb{N}^2}$ such that
\begin{equation*}
  (L_1,L_2),(L_3,L_4),\ldots\mid\nu\simiid\nu\;.
\end{equation*}
\citet{Janson:2017aa} refers to such a graph as being \emph{rank one} if ${\nu=\mu\otimes\mu}$ for some random probability measure on $\mathbb{N}$, which is just
\eqref{eq:cond:iid}. Thus, rank one edge exchangeable graphs are precisely those corresponding to
exchangeable random partitions via $\Phi$. The intersection of edge exchangeable and $(\alpha,T)$-graphs are precisely those $(\alpha,T)$-graphs that 
have exchangeable vertex assignments, in which case $\Pi$ is
an exchangeable Gibbs partition. That includes the case ${\Pi\sim\CRP(\alpha,\theta)}$ above, for which \cite{Crane:Dempsey:2016} call $G=\Phi(\Pi)$ the \emph{Hollywood model}. 
\begin{proposition}
  Let $G$ be a rank one edge exchangeable graph, and let $\mathbf{D}_{\infty}$ be the limiting degree proportions $n^{-1}(\deg_1(n),\deg_2(n),\dotsc)$. Then $\mathbf{D}_{\infty}$ is NTL if and only if $G$ is distributed as a $(\alpha,T)$-graph, where $T$ distributed as in \eqref{eq:egp:arrivals}, in which case $\mathbf{D}_{\infty}$ is distributed as in \eqref{eq:degree:proportions}.
\end{proposition}

The results of \cref{sec:deg:dist:fixed} specialize for $G=\Phi(\Pi)$ where $\Pi$ has law $\CRP(\alpha,\theta)$. In particular, consider conditioning on the first $r$ arrival times, rather than all arrival times. As \cref{prop:crp:limits} shows, the scaled degrees have the same basic structure as in \cref{sec:deg:dist:fixed}, but $\bbE[W_{r,\infty}]$ is captured by a single beta random variable.

\begin{proposition} \label{prop:crp:limits}
  Let $G$ be an $(\alpha,T)$-graph for fixed $\alpha\in(0,1)$ such that $T$ are the arrival times induced by a $\CRP(\alpha,\theta)$ partition process \eqref{eq:crp:arrivals}. Then for every $r\in \bbN_+$, conditioned on $T_1,\dotsc,T_r$,
  \begin{align*}
    n^{-1} (\deg_j(n))_{1\leq j \leq r} \mid T_1,\dotsc,T_r \xrightarrow[n\to\infty]{\text{\small a.s.}} (\xi_1,\dotsc,\xi_r) \;,
  \end{align*}
  where
  \begin{align}
    \bigl(\xi_j \bigr)_{1\leq j \leq r} \equdist  \bvar_{T_r - r\alpha,\theta + r\alpha} \cdot \bigl(\Psi_j\prod_{i=j+1}^r (1 - \Psi_i)\bigr)_{1\leq j \leq r} \;,
  \end{align}
  with $\Psi_1=1$, $\Psi_j \sim \BetaDist(1-\alpha,T_j - 1 - (j-1)\alpha)$ and $\bvar_{T_r - r\alpha, \theta + r\alpha}$ mutually independent random variables for $j \geq 1$. This implies the joint distributional identity
\begin{align}
  \gvar_{T_r  + \theta } \cdot \bigl(\xi_j \bigr)_{1\leq j \leq r} \equdist  \gvar_{T_r - r\alpha} \cdot \bigl( \Psi_j\prod_{i=j+1}^r (1 - \Psi_i)\bigr)_{1\leq j \leq r} \;.
\end{align}
and the marginal identities for all $j>1$
\begin{gather*}
  \xi_j \equdist \bvar_{1-\alpha,T_j - 1 + \theta + \alpha} \\
  \xi_{j+1} \mid \xi_j, \Delta_{j+1} \equdist \xi_j \bvar_{T_j + \theta, \Delta_{j+1}} \;.
\end{gather*}
\end{proposition}

The random variable $\xi_j$ is independent of $j$, given $T_j$. 
Among all $(\alpha,T)$-graphs derived from exchangeable Gibbs partitions, this property characterizes those derived from $\CRP(\alpha,\theta)$ partitions, and stems from the arrival time distribution (see \cref{sec:arrival:classification}).
\begin{proposition} \label{prop:crp:marginal}
  Let $G$ be an $(\alpha,T)$-graph such that $T$ are the arrival times induced by an exchangeable Gibbs partition $\Pi$. For any $j\geq 1$, the marginal distribution of $\xi_j$ conditioned on $j$th arrival time depends only on $T_j$ if and only if $\Pi$ has law $\CRP(\alpha,\theta)$.
\end{proposition}


\subsection{Graphs with exchangeable interarrival times}
\label{sec:examples:exch:interarrivals}

We next consider $(\alpha,T)$-graphs for which the interarrival times ${\Delta_j = T_j - T_{j-1}}$ are exchangeable. 
An immediate consequence of exchangeability is that ${k^{-1}\sum_{j\leq k}\Delta_j\rightarrow\mu}$ almost surely
for some constant ${\mu}$ in ${[1,\infty]}$. \cref{thm:limiting:degree:sequence} implies:
\begin{corollary}
  If ${\mu}$ is finite, the limiting degrees scale as $n^{1/\gamma}$ in \eqref{eq:limiting:degree:sequence}, where ${\gamma = (\mu - \alpha)/(\mu - 1)}$. If $\text{Var}(\Delta_j)$ is finite for all $j$, then the degree distribution converges to \eqref{eq:degree:distn:linear}.
\end{corollary}

Stronger results hold when the interarrivals are \iid geometric variables, corresponding to the Yule--Simon
model \cite{Yule:1925,Simon:1955}. Recall that a positive random variable $\mlvar_{\sigma}$ is said to have \emph{Mittag--Leffler distribution} with parameter $\sigma\in(0,1)$ if $\mlvar_{\sigma}=\mathcal{Z}_{\sigma}^{-\sigma}$, where $\mathcal{Z}_{\sigma}$ is a positive $\sigma$-stable random variable, characterized by the Laplace transform $\mathbb{E}[e^{-\lambda\mathcal{Z}}]=e^{-\lambda^{\sigma}}$ and density $f_{\sigma}(z)$. See \cite{Pitman:2006,James:2015aa} for details. Define $\mathcal{Z}_{\sigma,\theta}$ for $\theta > -\sigma$ as a random variable with the polynomially tilted density ${f_{\sigma,\theta}\propto z^{-\theta}f_{\sigma}(z)}$, and let $\mlvar_{\sigma,\theta}=\mathcal{Z}_{\sigma,\theta}^{-\sigma}$. We denote the law of $\mlvar_{\sigma,\theta}$ by $\MittagLeffler(\sigma,\theta)$, which is known as the generalized Mittag--Leffler distribution \citep{Pitman:2003aa,James:2015aa}.

\begin{proposition} \label{prop:ys:limit}
  Let $G$ be an $(\alpha,T)$-graph with $\alpha=0$, and $T$ constructed from \iid $\Geom(\beta)$ interarrival times, for $\beta\in(0,1)$. Then for every $r\in\bbN_+$, conditioned on $T_1,\dotsc,T_r$,
  \begin{align*}
    n^{-(1-\beta)} (\deg_j(n))_{1\leq j \leq r} \mid T_1,\dotsc,T_r 
    \quad\xrightarrow[n\to\infty]{\text{\small a.s.}}\quad
    (\xi_1,\dotsc,\xi_r) \;,
  \end{align*}
  where
  \begin{align}
    \bigl(\xi_j \bigr)_{1 \leq j \leq r} \equdist \mlvar_{1-\beta,T_r-1} \bvar_{T_r,(T_r-1)\frac{\beta}{1-\beta}} \cdot \bigl( \Psi_j \prod_{i=j+1}^r (1-\Psi_i)  \bigr)_{1 \leq j \leq r} \;,
  \end{align}
  with $M_{1-\beta,T_r-1}$, $\bvar_{T_r,(T_r-1)\frac{\beta}{1-\beta}}$, $\Psi_1=1$ and $\Psi_j\sim\BetaDist(1,T_j-1)$ mutually independent random variables for $j \geq 1$. This implies the joint distributional identity
  \begin{align*}
    \gvar_{T_r}^{1-\beta} \cdot \bigl(\xi_j \bigr)_{1 \leq j \leq r} \equdist 
    \gvar_{T_r} \cdot \bigl( \Psi_j \prod_{i=j+1}^r (1-\Psi_i)  \bigr)_{1 \leq j \leq r} \;,
  \end{align*}
  and the marginal identities for $j>1$
  \begin{gather}
    \xi_j \equdist \mlvar_{1-\beta} \bvar_{1,T_j-1}^{1-\beta} \\
    \xi_{j+1} \mid \xi_j, \Delta_{j+1} \equdist \xi_j \bvar^{1-\beta}_{T_j,\Delta_{j+1}} \\
    \xi_j \equdist \mlvar_{1-\beta,T_j-1} \bvar_{T_j,(T_j-1)\frac{\beta}{1-\beta}} \Psi_j \equdist
     \mlvar_{1-\beta,T_j-1} \bvar_{1,\frac{T_j - 1}{1-\beta}} \equdist
     \mlvar_{1-\beta, T_j} \bvar_{1,\frac{T_j - 1 + \beta}{1-\beta}} \\
    \xi_{j+1} \mid \xi_j, \Delta_{j+1} \equdist \xi_j \bvar_{\frac{T_j}{1-\beta},\frac{\Delta_{j+1}}{1-\beta}} \prod_{i=1}^{\Delta_{j+1}} \bvar_{\frac{T_j - 1 + i - \beta}{1-\beta},\frac{\beta}{1-\beta}} \\
    \xi_j \gvar_{T_j}^{1-\beta} \equdist \gvar_{1}  \;.
  \end{gather}
\end{proposition}

\Citet*{Pekoz:Rollin:Ross:2017aa} consider the following two-color P\'{o}lya urn: Let $\Delta_1,\Delta_2,\dotsc$ be drawn \iid from some distribution $P_{\Delta}$, and define $T_j = \sum_{i=1}^j \Delta_j$. Starting with $w$ white balls and $b$ black balls, at each step $n\neq T_j$, a ball is drawn and replaced along with another of the same color. On steps $n=T_j$, a black ball is added to the urn. Of interest is the distribution of the number of white balls in the urn after $n$ steps.

In the language of $(\alpha,T)$-graphs, consider a seed graph $G_{w+b}$ with $k_{w+b} < w + b$ vertices and $w + b$ edges arranged arbitrarily, the only constraint being that there exists a bipartition $\mathbf{V}_w \cup \mathbf{V}_b=\mathbf{V}(G_{w+b})$ so that the total degree of the vertices in $\mathbf{V}_w$ is $D_{w}(w+b) = w$, and of those in $\mathbf{V}_b$ is $D_{b}(w+b)=b$. For $T$ constructed from \iid interarrivals, $D_{w}(n)$ corresponds to the number of white balls after $n$ steps. For interarrivals drawn \iid from the geometric distribution, the following result characterizes the limiting distribution of $D_w(n)$, which was left as an open question by \citet*{Pekoz:Rollin:Ross:2017aa}.
\begin{proposition} \label{prop:immigration:urn}
  Let $D_{w}(n)$ be the number of white balls in the P\'{o}lya urn with immigration from {\rm\citep{Pekoz:Rollin:Ross:2017aa}} starting with $w$ white balls and $b$ black balls, where the immigration times have \iid $\Geom(\beta)$ distribution. Then
  \begin{align*}
    n^{-(1-\beta)}D_{w}(n) 
    \quad\xrightarrow[n\to\infty]{\text{\small a.s.}}\quad
    \xi_{w,w+b} 
    \;\equdist\;
    \bvar_{w,b} \bvar_{w+b,(w+b-1)\frac{\beta}{1-\beta}} \mlvar_{1-\beta,w+b-1} \;,
  \end{align*}
  which implies the distributional identities
  \begin{gather}
    \xi_{w,w+b} 
    \;\equdist\;
    \bvar_{w,\frac{(w-1)\beta + b}{1-\beta}} \mlvar_{1-\beta,w+b-1} \\
    \xi_{w,w+b} 
    \;\equdist\;
    \bvar_{w,\frac{w\beta + b}{1-\beta}} \mlvar_{1-\beta,w+b} \\
    \xi_{w,w+b} \gvar_{w+b}^{1-\beta} 
    \;\equdist\;
    \gvar_w \;.
  \end{gather}
\end{proposition}

\begin{table}
  \begin{center}
\resizebox{\textwidth}{!}{
\begin{tabular}{ccccc}
  ${\mathbb{P}\braces{n+1\text{ is arrival time}|G_{n}}}$
  &
  $\Phi^{-1}(G)$ is 
  &
  &
  &
  \\
  depends on 
  &
  $(\alpha,T)$-urn
  &
  $\mathcal{L}(\Phi^{-1}(G))$
  &
  $\mathcal{L}(G)$
  &
  $\mathcal{L}(\Delta_k)$\\
  \midrule
  $n$
  &
  yes
  &
  $\CRP(\theta)$
  &
  Hollywood model
  &
  \eqref{eq:crp:arrivals}\\
  $n,\#\text{vertices in }G_{n}$
  &
  yes
  &
  Gibbs partition
  &
  $\subset$ rank one edge exch.
  &
  \eqref{eq:egp:arrivals}\\
  $n$, $\#\text{vertices in }G_{n}$, degrees
  &
  no
  &
  exch. partition
  &
  rank one edge exch.
  &
  -- \\
  deterministic
  &
  yes
  &
  --
  &
  PA tree
  &
  $\delta_2$ \\
  independent
  &
  yes
  &
  Yule--Simon process 
  &
  ACL \citep{Aiello:Chung:Lu:2001,Aiello:Chung:Lu:2002}
  &
  $\Geom(\beta)$\\
  $n+1 - T_{k(n)}$
  &
  yes
  &
  $(\alpha,T)$-urn
  &
  $(\alpha,T)$-graph
  &
  \iid\\
  \bottomrule
\end{tabular}
}
  \end{center}
\caption{Classification of different models according to which statistics of $G_n$ determine the
  probability that a new vertex is observed at time $n+1$.}
\label{tab:igor}
\end{table}

\subsection{Classification by arrival time probabilities}
\label{sec:arrival:classification}

\citet*{deBlasi:etal:2015} classify exchangeable partitions according to the quantities on which the probability
of observing a new block in draw ${n+1}$ depends \citep[][Proposition 1]{deBlasi:etal:2015}, conditionally
on the partition observed up to time $n$.
This classification can be translated to random graphs via the induced partition ${\Phi^{-1}(G)}$, and
can be extended further since partitions induced by $(\alpha,T)$-graphs need not be exchangeable:
See \cref{tab:igor}.
One might also consider a sequence of interarrival distributions indexed by the number of vertices, yielding a bespoke generalization of the last row, where the probability of a new vertex depends on $n+1 - T_{k(n)}$ and the number of vertices.

 \section*{Acknowledgments} We are grateful to Nathan Ross for helpful comments on the manuscript. BBR is supported by the European Research Council under the European Union's Seventh Framework Programme (FP7/2007--2013) / ERC grant agreement no. 617071. PO was supported in part by grant FA9550-15-1-0074 of AFOSR.

\bibliography{references}
\bibliographystyle{abbrvnat}

\begin{appendices}
\crefalias{section}{appsec}
\crefalias{subsection}{appsec}


\section{Proofs}
\label{appx:proofs}

\subsection{Proof of \texorpdfstring{\cref{theorem:sb}}{Theorem 2}}
\label{sec:proof:representation}

\begin{proof}
	Let $G(\alpha,T)=((L_1,L_2),(L_3,L_4),\dotsc)$ be an $(\alpha,T)$-graph, and let $k_n$ denote the number of vertices after $n$ edge ends have been drawn. For notational convenience, let $D_j(n) := \deg_j(n)$ be the degree of the $j$th vertex at step $n$. Then from the sequence of laws $P_{\alpha}(k;l_1,\dotsc,l_j)$ in \eqref{eq:p:alpha}, the probability of a particular sequence is 
	\begin{align*}
		\bbP[G_{n/2}(\alpha,T) \mid T=t ] = \frac{1}{\Gamma(n - k_n\alpha)} \prod_{j=1}^{k_n} \frac{\Gamma(t_j - j\alpha) \Gamma(\ct_j(n) - \alpha)}{\Gamma(t_j - 1 - (j-1)\alpha + \delta_1(j)) \Gamma(1 - \alpha)} \;.
	\end{align*}
	Now let $H_n(\alpha,T)$ denote the first $n$ elements of the sequence $H(\alpha,T)$, so that
	\begin{align*}
		\bbP[H_{n/2}(\alpha,T), (\Psi_j)_{2 \leq j \leq k_n} \mid T=t] = \prod_{i=2}^{k_n} \frac{\Gamma(t_j - j\alpha) \Psi_j^{\ct_j(n) - \alpha - 1} (1-\Psi_j)^{\bar{\ct}_{j-1}(n) - (j-1)\alpha - 1} }{\Gamma(1-\alpha)\Gamma(t_j - 1 - (j-1)\alpha)} \;.
	\end{align*}
	Hence, by marginalizing $(\Psi_j)_{2 \leq j \leq k_n}$, it follows that
	\begin{align*}
		\bbP[H_{n/2}(\alpha,T) \mid T=t ] &= \int_{[0,1]^{k_n}} \bbP[H_{n/2}(\alpha,T), (\Psi_j)_{2 \leq j \leq k_n} \mid T=t] \; d\Psi_2 \dotsm d\Psi_{k_n} \\
			&= \bbP[G_{n/2}(\alpha,T) \mid T=t ] \;.
	\end{align*}
	The equality is true for all $n\in\mathbb{N}_+$, and $t\in\mathbb{T}$, implying that ${\bbP[G(\alpha,T)]=\bbP[H(\alpha,T)]}$.
\end{proof}

\subsection{Proof of \texorpdfstring{\cref{prop:tail:product:convergence}}{Proposition 5}}
\label{sec:proof:tail:product:convergence}

The proof will make repeated use of the expectation of $W_{j,k}$ as $k\to\infty$. Observe that
\begin{align*}
	- \log \bbE[W_{j,k}] & = \sum_{i=j+1}^k \log \biggl( 1 + \frac{1-\alpha}{t_i - 1 - (i-1)\alpha} \biggr)
	 = \sum_{i=j+1}^k \frac{1-\alpha}{t_i - 1 - (i-1)\alpha} + O(i^{-2})
\end{align*}
By assumption in the linear regime, $t_i - 1 - (i-1)\alpha$ is well-approximated by $(i-1)(\mu_1-\alpha)$ for $i$ large enough, resulting in a finite error $C_{j,k}$ that converges as $k\to\infty$. Hence,
\begin{align*}
	- \log \bbE[W_{j,k}] & = C_{j,k} + \frac{1-\alpha}{\mu_1-\alpha} \sum_{i=j}^{k-1} i^{-1} + O(i^{-2}) = C_{j,k} + \frac{1-\alpha}{\mu_1-\alpha}\log\biggl(\frac{k-1}{j}\biggr) + O(j^{-1}) \;,
\end{align*}
and therefore for all $j>1$,
\begin{align} \label{eq:linear:w:expectation}
	\bbE[W_{j,k}] = A_j \biggl( \frac{j}{k-1}  \biggr)^{\frac{1-\alpha}{\mu_1 - \alpha}} \quad \text{as} \quad k\to\infty \;.
\end{align}
Clearly, this converges to zero for each $j$. 
Similarly, in the sub-linear regime,
\begin{align*}
	- \log \bbE[W_{j,k}] & = C'_{j,k} + \sum_{i=j}^{k-1} \frac{1-\alpha}{\mu_{\sigma} i^{1/\sigma} - i\alpha} + O(i^{-2}) \\
	& = C'_{j,k} + \frac{\sigma(1-\alpha)}{\alpha(1-\sigma)} \log \biggl( \frac{\mu_{\sigma} - \alpha (k-1)^{-\frac{1-\sigma}{\sigma}}}{\mu_{\sigma} - \alpha j^{-\frac{1-\sigma}{\sigma}}} \biggr) + O(j^{-1}) \;.
\end{align*}
Since $\frac{1-\sigma}{\sigma} > 0$,
\begin{align} \label{eq:sublinear:w:expectation}
	\bbE[W_{j,k}] = A'_j \biggl(1 - \frac{\alpha}{\mu_{\sigma} j^{\frac{1-\sigma}{\sigma}}}\biggr)^{\frac{\sigma(1-\alpha)}{\alpha(1-\sigma)}} \biggl(1 - \frac{\alpha}{\mu_{\sigma} (k-1)^{\frac{1-\sigma}{\sigma}}} \biggr)^{-\frac{\sigma(1-\alpha)}{\alpha(1-\sigma)}} \quad \text{as} \quad k\to\infty \;.
\end{align}
Clearly, this converges to something non-zero for each $j$.

\begin{proof}[Proof of \cref{prop:tail:product:convergence}]
	Define $M_{j,k}:=W_{j,k}/\bbE[W_{j,k}]$. Since ${\bbE[ M_{j,k+1} \mid M_{j,k} ]= M_{j,k}}$ and $\bbE[M_{j,j+1}=1]$, $M_{j,k}$ is a nonnegative martingale with mean 1 for $k > j$; it therefore converges almost surely to a random variable $M_{j,\infty}$. Hence,
	\begin{align*}
		W_{j,\infty}:= \lim_{k\to\infty} W_{j,k} = M_{j,\infty} \lim_{k\to\infty} \bbE[W_{j,k}] \;.
	\end{align*}
	Simple calculations show that $\bbE[M^2_{j,k}] = 1 + O(k^{-1})$, and therefore $M_{j,k}$ is bounded in $L_2$ and also in $L_1$. Hence, $W_{j,\infty}$ exists if and only if $\lim_{k\to\infty} \bbE[W_{j,k}]$ exists. By \eqref{eq:linear:w:expectation} and \eqref{eq:sublinear:w:expectation}, that is the case.

	By Markov's inequality, for any $\lambda > 0$,
	\begin{align*}
	\bbP[ W_{j,k} \geq \lambda ] \leq \frac{1}{\lambda} \bbE[W_{j,k}]
		\leq
		\begin{cases}
			O(k^{-\frac{1-\alpha}{\mu_1 - \alpha}}) \quad \text{for} \quad \sigma=1 \\
			O(\lambda^{-1})  \quad \text{for} \quad \sigma < 1
		\end{cases} \;,
	\end{align*}
	indicating that $W_{j,\infty} = 0$ in the linear case. On the other hand,
	\begin{align*}
		\bbP[ W_{j,k} \leq \lambda ] & = \bbP[ W_{j,k}^{-1} \geq \lambda^{-1} ] \leq \lambda \bbE[W_{j,k}^{-1}] = \lambda \bbE[W_{j,k}]^{-1} \prod_{i=j+1}^k \frac{1 - \frac{1}{t_i - i\alpha}}{1 - \frac{1}{t_i - 1 - (i-1)\alpha}} \\
		& \leq
		\begin{cases}
			O(k^{\frac{1-\alpha}{\mu_1 - \alpha}}) \quad \text{for} \quad \sigma=1 \\
			O(\lambda)  \quad \text{for} \quad \sigma < 1
		\end{cases} \;,
	\end{align*}
	indicating that $W_{j,\infty} > 0$ in the sub-linear case. 
\end{proof}

\subsection{Proofs for \texorpdfstring{\cref{sec:deg:dist:fixed}}{Section 3.2}}
\label{sec:proof:degree:sequence}

The proof of \cref{thm:limiting:degree:sequence} uses martingale methods adapted and extended from those used by \citet{Mori:2005}. See also \cite{Hofstad:2016,Durrett:2006}. For notational convenience, we let $\ct_{j,n} = \ct_j(n)$. Fix $r\in\bbN_+$ and $p = (p_1,p_2,\dotsc)$ such that $p_j > -(1-\alpha)$ for each $1\leq j \leq r$, and $p_j=0$ for $j>r$. Let $\bar{p}_j=\sum_{i=1}^j p_i$, with $\bar{p}_0=0$ and $\bar{p}_j=\bar{p}_r:=\bar{p}$ for $j>r$. For fixed $t\in\mathbb{T}$, define
\begin{align} \label{eq:martingale:z}
	& Z_n(p,t) := \\
	& \frac{\Gamma(n - k_n \alpha)}{\Gamma(n - k_n \alpha + \bar{p}_{k_n})} 
	  \prod_{j=1}^{r} \frac{\Gamma(\ct_{j,n} - \alpha + p_j)}{\Gamma(\ct_{j,n} - \alpha)}
	   \prod_{k=r+1}^{k_n} \frac{\Gamma(t_k - 1 - (j-1)\alpha)\Gamma(t_k - j\alpha + \bar{p}_k)}{\Gamma(t_k - 1 - (k-1)\alpha + \bar{p}_k)\Gamma(t_k - k\alpha)} \nonumber \;.
\end{align}
The asymptotic behavior of $Z_n(p,t)$ is described by the following two lemmas, from which \cref{thm:limiting:degree:sequence} follows.

\begin{lemma} \label{lemma:degree:sequence:martingale}
	Let $Z_n(p,t)$ be as in \eqref{eq:martingale:z}, with $p_j > -(1-\alpha)$ for each $j\geq 1$. Then $Z_n(p,t)$ is a nonnegative martingale with respect to $(\mathcal{F}_n)_{n\geq T_r}$, the filtration generated by the $(\alpha,t)$ sampling process, for $n\geq T_r$, and therefore
	\begin{align*}
		Z_n(p,t) \xrightarrow[n\to\infty]{\text{\small a.s.}} \xi_1^{p_1} \dotsm \xi_r^{p_r} \;.
	\end{align*}
	Furthermore, if $p_j > -(1-\alpha)/2$ for each $j\geq 1$, then $Z_n(p,t)$ converges in $L_2$ and therefore also in $L_1$.
\end{lemma}
\begin{proof}
	$Z_n(p,t)$ is nonnegative by construction. Furthermore,
	\begin{align*}
		& \bbE\biggl[ \frac{\Gamma(n+1 - k_{n+1} \alpha)}{\Gamma(n+1 - k_{n+1} \alpha + \bar{p}_{k_{n+1}})} \prod_{j=1}^{k_{n+1}} \frac{\Gamma(\ct_{j,n+1} - \alpha + p_j)}{\Gamma(\ct_{j,n+1} - \alpha)} \mid \mathcal{F}_{n} \biggr] := \bbE[ R_{n+1}(p) \mid \mathcal{F}_n] \\
		& \; = R_{n}(p) \biggl[ \indicator\{ t_{k_n+1} > n+1 \} + \frac{\Gamma(n -k_n\alpha + \bar{p}_{k_n})\Gamma(n+1 - (k_n+1)\alpha)}{\Gamma(n+1 - (k_n+1)\alpha + \bar{p}_{k_n+1})\Gamma(n - k_n\alpha)}\indicator\{ t_{k_n+1}=n+1 \} \biggr] \;,
	\end{align*}
	from which it follows that $\bbE[Z_{n+1}(p,t)\mid \mathcal{F}_n] = Z_n(p,t)$ for all $n\geq T_r$. Furthermore, $\bbE[Z_n(p,t)] < \infty$ and does not depend on $n$. To bound $Z_n(p,t)$ in $L_2$, observe that by the properties of the gamma function \cite{Tricomi:Erdelyi:1951},
	\begin{align} \label{eq:c:asymptotics}
		c_{\bar{p}}:=\frac{\Gamma(n - k_n\alpha)}{\Gamma(n - k_n\alpha + \bar{p}_{k_n})} = n^{-\bar{p}}(1 + O(n^{-1})) \quad \text{as} \quad n\to\infty \;,
	\end{align}
	which implies that $c_{\bar{p}}^2/c_{2\bar{p}}\to 1$. Using the fact that $x \mapsto \Gamma(x+y)/\Gamma(x)$ is increasing in $x$ for $y\geq 0$, it follows that $Z_n(p,t)^2\leq Z_n(2p,t)$ (see \citep[][Section 8.7]{Hofstad:2016} for a similar argument). Hence, ${\bbE[Z_n(p,t)^2]\leq \bbE[Z_n(2p,t)] < \infty}$ for $p_j > -(1-\alpha)/2$.
\end{proof}

\begin{lemma} \label{lemma:z:asymptotics}
	Let $Z_n(p,t)$ be as in \eqref{eq:martingale:z}, with $p_j > -(1-\alpha)$ for each $j\geq 1$. Then there exists a constant $M_{\bar{p}}(t)$ such that
	\begin{align*}
		Z_n(p,t) = M_{\bar{p}}(t) n^{-\bar{p}/\gamma} \prod_{j=1}^r \ct_{j,n}^{p_j}(1 + O(n^{-1})) \quad \text{as} \quad n\to\infty \;.
	\end{align*}
\end{lemma}
\begin{proof}
	The properties of the gamma function show that for each $j\leq r$,
	\begin{align*}
		\frac{\Gamma(\ct_{j,n}-\alpha + p_j)}{\Gamma(\ct_{j,n})} = \ct_{j,n}^{p_j}(1 + O(n^{-1})) \quad \text{as} \quad n\to\infty \;.
	\end{align*}
	Hence, it suffices to examine the behavior of the product
	\begin{align*}
		M_n(p,t) & := \frac{\Gamma(n - k_n\alpha)}{\Gamma(n - k_n\alpha + \bar{p}_{k_n})}\frac{\Gamma(1 - \alpha + p_1)}{\Gamma(1 - \alpha)} 
	   \prod_{j=2}^{k_n} \frac{\Gamma(t_j - 1 - (j-1)\alpha)\Gamma(t_j - j\alpha + \bar{p}_j)}{\Gamma(t_j - 1 - (j-1)\alpha + \bar{p}_j)\Gamma(t_j - j\alpha)} \\
	   & = \prod_{m=1}^{n-1} \frac{m - k_m\alpha}{m - k_m\alpha + \bar{p}_{k_m} \indicator\{ t_{k_m+1} \neq m+1 \}} \\
	   & = \biggl( \prod_{m=1}^{n-1} \frac{m - k_m\alpha}{m - k_m\alpha + \bar{p}_{k_m}} \biggr) \biggl( \prod_{j=1}^{k_n - 1} \frac{t_{j+1} - 1 - j\alpha + \bar{p}_{j}}{t_{j+1} - 1 - j\alpha} \biggr) := M_n^{(1)}(p,t) M_n^{(2)}(p,t) \;.
	\end{align*}
	Taking the logarithm, the first term is
	\begin{align*}
		\ln M_n^{(1)}(p,t) & = \sum_{m=1}^{n-1} \ln(m - k_m\alpha) - \ln(m - k_m\alpha + \bar{p}_{k_m}) \\
		& = \sum_{m=1}^{n-1} \ln\bigl( 1 + \frac{\bar{p}_{k_m}}{m - k_m\alpha} \bigr) = C_1 - \sum_{m=1}^{n-1} \frac{\bar{p}_{k_m}}{m - k_m\alpha} + O(m^{-2}) \;,
	\end{align*}
	where $C_1$ captures the error in the approximation for terms where $\bar{p}_{k_m} > m - k_m\alpha$, which is finite because $p_j=0$ for $j>r$. By assumption when $\mu < \infty$, for any $\epsilon > 0$ there exists some finite $K$ such that for all $j\geq K$, $\lvert k_m - m/\mu \rvert \leq \epsilon$. Therefore, absorbing the additional error into $C_1$,
	\begin{align*}
		\ln M_n^{(1)}(p,t) & = C_1 - \sum_{m=1}^{n-1} \frac{ \bar{p} }{m(1 - \alpha/\mu)} + O(m^{-2}) = C_1 - \frac{\bar{p}\mu}{\mu - \alpha}\ln n + O(n^{-1}) \;.
	\end{align*}
	In the sub-linear regime ($\mu=\infty$), the second term is $-\bar{p} \ln n$. 
	Similarly, for the second term of $M_n(p,t)$,
	\begin{align*}
		\ln M_n^{(2)}(p,t) & = \sum_{j=1}^{k_n - 1} \ln\bigl( 1 + \frac{\bar{p}_j}{t_{j+1} - 1 - j\alpha} \bigr) = C_2 + \sum_{j=1}^{k_n - 1} \frac{\bar{p}_j}{t_{j+1} - 1 - j\alpha} + O(j^{-2}) \\
		& = C_2 + \sum_{j=1}^{k_n - 1} \frac{\bar{p}}{j(\mu - \alpha)} + O(j^{-2}) = C_2 + \frac{\bar{p}_j}{\mu - \alpha} \ln n + O(n^{-1}) \;.
	\end{align*}
	In the sub-linear regime, the second term is also $O(n^{-1})$.
	Hence, letting $M_{\bar{p}}(t)=e^{C_1 + C_2}$, the result follows.
\end{proof}

\begin{proof}[Proof of \texorpdfstring{\cref{thm:limiting:degree:sequence}}{Theorem 6}]
	Equations \eqref{eq:limit:seq:full} and \eqref{eq:limiting:degree:sequence} follow from \cref{lemma:degree:sequence:martingale,lemma:z:asymptotics}. To establish \eqref{eq:degree:sequence:g}, let $p\setminus p_j$ be $p$ with the $j$th element set to zero, observe that
	\begin{align*}
		& \bbE [\zeta_1^{p_1} \dotsm\zeta_r^{p_r}] = \bbE[Z_{t_r}(p,t)] \\
		& = \bbE[Z_{t_r-1}(p\setminus p_r,t)] \frac{ \Gamma(1 - \alpha + p_r) \Gamma(t_r - r\alpha) \Gamma(t_r - 1 - (r-1)\alpha + \bar{p}_{r-1}) }{ \Gamma(1-\alpha) \Gamma(t_{r} - 1 - (r-1)\alpha) \Gamma(t_{r} - r\alpha + \bar{p}_r) } \\
		& = \prod_{j=2}^r \frac{ \Gamma(1 - \alpha + p_j) \Gamma(t_j - j\alpha) \Gamma(t_j - 1 - (j-1)\alpha + \bar{p}_{j-1}) }{ \Gamma(1-\alpha) \Gamma(t_{j} - 1 - (j-1)\alpha) \Gamma(t_{j} - j\alpha + \bar{p}_j) }= \prod_{j=2}^r \bbE[\Psi_j^{p_j} (1-\Psi_j)^{\bar{p}_{j-1}} ] \;.
	\end{align*}
\end{proof}

\subsection{Proofs for \texorpdfstring{\cref{sec:degree:distributions}}{Section 3.4}}
\label{sec:proof:degree:distributions}

In the sub-linear regime, the degree distribution follows from results on exchangeable random partitions and can be read from \cite[Lemma 3.11]{Pitman:2006}. The proof of \cref{theorem:degree:distn} in the linear case relies on the following lemma, which says that for large enough $j$, $W_{j,k}$ is well-approximated by $(j/k)^{1-1/\gamma}$ for all $k > j$, where $\gamma = (\mu-\alpha)/(\mu-1)$.
\begin{lemma} \label{lemma:pp:1}
	Let $G$ be a random $(\alpha,T)$ graph for some $\alpha\in(-\infty,1)$ and $T\in\mathbb{T}$. Assume that $\text{Var}(\Delta_j)<V^2_{\Delta}$ for all $j\in\bbN_+$, that $\bbE[\Delta_j]=\mu$ for all $j > 1$, and that $\lvert\text{Cov}(\Delta_i,\Delta_j)\rvert = C^2_{\Delta}\lvert i-j\rvert^{-\ell_{\Delta}}$ for all $i,j > 1$, some $C^2_{\Delta}\geq 0$, and some $\ell_{\Delta}>0$. 
	Let $W_{j,k}$ be as in \eqref{eq:sb:2}. Then for every $\lambda \geq 0$, there exists some $K < \infty$ that does not depend on $n$ such that for all $k\geq K$,
	\begin{align*}
		\bbP\left[ \sup_{k+1\leq m \leq n} \left\lvert W_{j,m} \left( \frac{m}{k} \right)^{1-1/\gamma} - 1  \right\rvert \leq \lambda \right] \geq 1 - \lambda \;.
	\end{align*}
\end{lemma}
\begin{proof}
	Condition on $T=t$ and define for fixed $j$
	\begin{align*}
		M_{j,k} := W_{j,k}\prod_{j=k+1}^k \frac{T_j - j\alpha}{T_j - 1 - (j-1)\alpha} \;,
	\end{align*}
	which is a martingale with mean 1 for $k\geq j+1$. By Doob's maximal inequality \cite[\egc][Chapter II]{Revuz:Yor:1999}, for any $c > 0$ and $j \geq 1$,
	\begin{align} \label{eq:doob:1}
		\bbP\bigl[ \sup_{j+1\leq m \leq k} \lvert M_{j,m} - 1 \rvert \geq c \mid T \bigr] \leq \frac{1}{c^2} \bbE[ (M_{j,k} - 1)^2 \mid T] \;.
	\end{align}
	The right-hand side is 
	\begin{align} \label{eq:rhs:bound}
		\frac{1}{c^2} \bbE[ (M_{j,k} - 1)^2 \mid T] 
			& = \frac{1}{c^2} \biggl(\frac{\bbE[W_{j,k}^2 \mid T]}{\bbE[ W_{j,k} \mid T]^2} - 1 \biggr)
			 \leq \frac{1}{c^2} \frac{1}{1-\alpha} \frac{1}{j} + O(j^{-2}) \;.
	\end{align}
	The product in $M_{j,k}$ is
	\begin{align*}
		&\exp \biggl( \sum_{i=j+1}^k \log\biggl( 1 + \frac{1-\alpha}{T_i - 1 -(i-1)\alpha}  \biggr)  \biggr)
			 = \exp\biggl( \sum_{i=j+1}^k \frac{1-\alpha}{T_i - 1 -(i-1)\alpha} + O(i^{-2}) \biggr) \\
			& = \biggl( \frac{k}{j} \biggr)^{1-1/\gamma} e^{-\epsilon_{j,k}} \bigl( 1 + O(j^{-1}) \bigr)
			\quad \text{where} \quad
			\epsilon_{j,k} := \sum_{i=j+1}^k \frac{(1-\alpha)}{(i-1)(\mu-\alpha)} - \frac{(1-\alpha)}{T_i -1 - (i-1)\alpha} \;.
	\end{align*}
	Define $\bar{\epsilon}_{j,k}$ to be sum of the absolute values of the terms in $\epsilon_{j,k}$. Because it is a running sum of nonnegative terms, $\bar{\epsilon}_{j,k}$ is a nonnegative submartingale. Therefore, another application of Doob's maximal inequality followed by Jensen's inequality shows that
	\begin{align*}
		\bbP[ \sup_{j+1\leq m \leq k} \lvert \epsilon_{j,m} \rvert \geq c_{\epsilon} ] 
			& \leq \bbP[ \sup_{j+1\leq m \leq k} \bar{\epsilon}_{j,m} \geq c_{\epsilon}] \\
			& \leq \frac{\bbE[\bar{\epsilon}_{j,k}]}{c_{\epsilon}} 
				\leq \frac{1}{c_{\epsilon}(\mu-\alpha)}\sum_{i=j+1}^k \frac{1}{j^2} \bbE[ \lvert T_j - (j-1)\mu - 1 \rvert ] \\
			& \leq \frac{1}{c_{\epsilon}(\mu-\alpha)}\sum_{i=j+1}^k \frac{1}{j^2} \biggl(\text{Var}\biggl(\sum_{i=2}^j (\Delta_i - \mu)\biggr)\biggr)^{1/2} \\
			& \leq \frac{1}{c_{\epsilon}(\mu-\alpha)}\sum_{i=j+1}^k \frac{1}{j^2} (V_{\Delta}^2 (j + C_{\Delta}^2 j^{2-\ell_{\Delta}}) )^{1/2} \\
			& \leq \frac{\sqrt{V^2_{\Delta} + C^2_{\Delta}}}{c_{\epsilon}(\mu-\alpha)} \biggl( \frac{2/\ell_+}{j^{\ell_+/2}} - \frac{2/\ell_+}{k^{\ell_+/2}} + O(j^{-(1+\ell_+/2)}) \biggr) \;,
	\end{align*}
	where $\ell_+ := \min\{\ell_{\Delta},1\}$. 
	Let $A_{j,k}^{\delta}$ be the event that $\lvert \epsilon_{j,m} \rvert < j^{-\delta}$ for each $j+1\leq m\leq k$ and some fixed $0 < \delta < \ell_+/2$. Then
	\begin{align*}
		\bbP[A_{j,k}^{\delta}] & = \bbP[ \sup_{j+1\leq m \leq k} \lvert \epsilon_{j,m} \rvert < j^{-\delta} ] \geq \biggl( 1 - \frac{2\sqrt{V^2_{\Delta} + C^2_{\Delta}}}{\ell_+(\mu-\alpha)} \bigl( \frac{1}{j^{\ell_+/2-\delta}} + O(j^{-(1+\ell_+/2-\delta)}) \bigr)\biggr)\;.
	\end{align*}
	Hence, the probability of $A_{j,k}^{\delta}\to 1$ as $j\to\infty$. Now,
	taking expectation with respect to $T$ of the left-hand side of \eqref{eq:doob:1} yields
	\begin{align*}
		& \bbE_T\bigl[\bbP\bigl[ \sup_{k+1\leq m \leq k} \lvert W_{j,k} \bigl( \frac{m}{j} \bigr)^{1-1/\gamma}(1 + \epsilon_{j,m} + O(j^{-1})) - 1 \rvert \geq c \mid T \bigr] \bigr] \\
			& \quad = \bbP\bigl[ \sup_{j+1\leq m \leq k} \lvert W_{j,k} \bigl( \frac{m}{j} \bigr)^{1-1/\gamma} - 1 \rvert \geq c \bigr] \dotsb \\
			& \quad\quad \times \bbE_T\left[ 
				\indicator_{A^{\delta}_{j,k}} 
				\frac{\bbP\bigl[ \sup_{j+1\leq m \leq k} \lvert W_{j,k} \bigl( \frac{m}{j} \bigr)^{1-1/\gamma}(1 + \epsilon_{j,m} + O(j^{-1})) - 1 \rvert \geq c \mid T \bigr]}{\bbP\bigl[ \sup_{j+1\leq m \leq k} \lvert W_{j,k} \bigl( \frac{m}{j} \bigr)^{1-1/\gamma} - 1 \rvert \geq c \bigr]}   
				\right] \dotsb \\
			& \quad\quad + \bbE_T\bigl[ (1-\indicator_{A^{\delta}_{j,k}}) \bbP\bigl[ \sup_{j+1\leq m \leq k} \lvert W_{j,k} \bigl( \frac{m}{j} \bigr)^{1-1/\gamma}(1 + \epsilon_{j,m} + O(j^{-1})) - 1 \rvert \geq c \mid T \bigr] \bigr] \;.
	\end{align*}
	Observe that as $j\to\infty$, the last term converges to zero by monotone convergence, while the ratio inside the expectation converges to one on sets where $A_{j,k}^{\delta}$ obtains, which have asymptotic measure equal to one. The result follows by monotone convergence and comparison with \eqref{eq:rhs:bound}.
\end{proof}

The following lemma gives an estimate based on the beta-gamma algebra, and says that for large enough $j$, $\Psi_j$ can be approximated by $\gvar_{1-\alpha}/((j-1)\mu - 1 - j\alpha)$.
\begin{lemma} \label{lemma:pp:2}
	For $j < k$, let $\Psi_j$ and $\Psi_k$ be independent beta random variables as in \eqref{eq:sb:1}, and let $\gvar^{(j)}$ and $\gvar^{(k)}$ be independent $\GammaDist(1-\alpha,1)$ random variables as in \eqref{eq:recursion}. Then there exists some $J < \infty$ such that for every $j\geq J$ and $\lambda > 0$,
	\begin{align*}
		\bbP\left[ \sup_{j\leq k \leq n} \left\lvert \Psi_j \Psi_k - \frac{\gvar^{(j)} \gvar^{(k)}}{(j(\mu-\alpha) - (\mu-1))(k(\mu-\alpha)-(\mu-1))} \right\rvert < \lambda \right] \geq 1 - \lambda \;.
	\end{align*}
\end{lemma}
\begin{proof}
	Denote the quantity of interest as $\epsilon^{\Psi}_{j,k}$, and let $\Sigma_j:= \sum_{i=1}^j \gvar^{(i)} + \sum_{i=1}^{j-1} \gvar_{\Delta_{i+1}-1}$. Then using the distributional identity in \eqref{eq:recursion},
	\begin{align*}
		& \lvert \epsilon^{\Psi}_{j,k} \rvert = \biggl\lvert \Psi_j\Psi_k \frac{(j(\mu-\alpha)-(\mu-1))(k(\mu-\alpha)-(\mu-1)) - \Sigma_j\Sigma_k}{(j(\mu-\alpha)-(\mu-1))(k(\mu-\alpha)-(\mu-1))} \biggr\rvert \\
		& = \frac{\Psi_j\Psi_k}{(j(\mu-\alpha)-(\mu-1))(k(\mu-\alpha)-(\mu-1))} \dotsb \\
		& \quad \times \biggl\lvert \sum_{i=1}^j \sum_{m=1}^k \bigl(\gvar^{(i)}\gvar^{(m)} - (1-\alpha)^2\bigr) + \sum_{i=1}^j \sum_{m=1}^{k-1} \bigl(\gvar^{(i)}\gvar_{\Delta_{m+1}-1} - (1-\alpha)(\mu-1)\bigr) \biggr. \dotsb \\
		& \quad \biggl. + \sum_{i=1}^{j-1} \sum_{m=1}^{k} \bigl(\gvar_{\Delta_{j+1}-1}\gvar^{(m)} - (1-\alpha)(\mu-1)\bigr) + \sum_{i=1}^{j-1} \sum_{m=1}^{k-1} \bigl(\gvar_{\Delta_{j+1}-1} \gvar_{\Delta_{m+1}-1} - (\mu-1)^2 \bigr) \biggr\rvert \;.
	\end{align*}
	Denote by $\bar{\epsilon}^{\Psi}_{j,k}$ the term-wise absolute value of $\epsilon^{\Psi}_{j,k}$. As a cumulative sum of nonnegative terms, $\bar{\epsilon}^{\Psi}_{j,k}$ is a nonnegative submartingale for $k > j$, and hence by Doob's maximal inequality, for any $c>0$,
	\begin{align*}
		\bbP[ \sup_{j\leq k \leq n} \lvert \epsilon^{\Psi}_{j,k} \rvert \geq c] 
			& \leq \bbP[ \sup_{j\leq k \leq n} \bar{\epsilon}^{\Psi}_{j,k} \geq c] \leq \frac{1}{c}\bbE[\bar{\epsilon}^{\Psi}_{j,n}] \;.
	\end{align*}
	The independence properties of the beta and gamma random variables, along with simple but tedious calculations yields
	\begin{align*}
		\bbE[\bar{\epsilon}^{\Psi}_{j,n}] \leq \frac{C(\alpha,\mu,V^2_{\Delta}) + O(j^{-1})}{(j(\mu-\alpha)-(\mu-1))(n(\mu-\alpha)-(\mu-1))} \;,
	\end{align*}
	where $C(\alpha,\mu,V^2_{\Delta})$ is a constant that depends only on the model parameters. The result follows.
\end{proof}

Finally, the proof of \cref{theorem:degree:distn} requires that for large enough $j$, the error term
\begin{align*}
	\epsilon^r_{j,k} := \sum_{r=j}^k \frac{1}{2r^{2-2/\gamma}}(\Delta_{r+1} - \mu)
\end{align*}
is small with high probability, which is established by the following lemma.
\begin{lemma} \label{lemma:pp:3}
	For any $c>0$,
	\begin{align*}
		\bbP[\sup_{j\leq k \leq n} \lvert \epsilon^r_{j,k} \rvert < c] \geq 1 - \frac{V_{\Delta}}{c} \frac{1}{j^{1-2/\gamma}}(1 + O(j^{-2-2/\gamma})) \;.
	\end{align*}
\end{lemma}
\begin{proof}
	Denote by $\bar{\epsilon}^r_{j,k}$ the term-wise absolute value of $\epsilon^r_{j,k}$. $\bar{\epsilon}^r_{j,k}$ is a nonnegative submartingale for $k>j$, and an application of Doob's maximal inequality yields the result.
\end{proof}

\begin{proof}[Proof of \cref{theorem:degree:distn}]
	Let $U_0$ be a $\Uniform[0,1]$ random variable, and define $j_0=\lceil U_0 k_n \rceil$ to be a vertex selected uniformly at random from $G_n$, which has $k_n$ vertices. 
	For simplicity, assume that either all $T_j$ are even for $j>1$ (which is the case for graphs that are a.s. connected for all $n$), or that $T_j$ has equal probability of being odd or even for all $j>1$. (These assumptions are not necessary, but greatly simplify notation. More general assumptions may be accommodated.) Conditioned on $(\Psi_j)_{j>1}$, the expected number of directed edges from vertex $j_0$ to any vertex $j\in\bbN_+$ is
	\begin{align} \label{eq:num:directed:edges}
		\bbE[N_{j_0\to j,n}] & = \sum_{r=j_0\vee j}^{k_n} P_{j_0\to j,r} = \sum_{r=j_0\vee j}^{k_n} \Psi_{j_0} W_{j_0,r} \Psi_j W_{j,r} (\Delta_{r+1}/2 - 1)  \\
		& \quad\quad + \indicator\{ j_0 \geq j, T_{j_0} \text{ odd} \}\Psi_j W_{j,j_0}
			+ \indicator\{ j_0 < j, T_{j} \text{ even} \}\Psi_{j_0} W_{j_0,j-1} \nonumber \;.
	\end{align}
	Define
	\begin{align*}
		\hat{P}_{j_0\to j,r} := \frac{\gvar^{(j_0)} \gvar^{(j)}(\mu/2 - 1)}{j_0 j(\mu-\alpha)^2} \biggl( \frac{j_0}{r} \frac{j}{r}  \biggr)^{1-1/\gamma} \;.
	\end{align*}
	By \cref{lemma:pp:1,lemma:pp:2,lemma:pp:3}, for $n$ large enough, with probability at least $1-2\lambda$, 
	\begin{align*}
		(1-\lambda)\sum_{r=j_0\vee j}^{k_n} \hat{P}_{j_0\to j,r} \leq \sum_{r=j_0\vee j}^{k_n} P_{j_0\to j,r} \leq (1+\lambda)\sum_{r=j_0\vee j}^{k_n} \hat{P}_{j_0\to j,r} \;.
	\end{align*}
	Similar approximations hold for the final two terms of \eqref{eq:num:directed:edges}. 
	Let $j_+=j_0\vee j$. Now summing over $r$ yields
	\begin{align*}
		\bbE[\hat{N}_{j_0\to j,n}] = \frac{\gvar^{(j_0)} \gvar^{(j)}(\mu/2 - 1)}{(\mu-\alpha)^2(1-2/\gamma)} \frac{1}{k_n} \biggl( \frac{j_0}{k_n} \frac{j}{k_n} \biggr)^{-1/\gamma} \biggl( \biggl(\frac{k_n}{j_+}\biggr)^{1-2/\gamma} - 1 \biggr)\biggl(1 + O(j_+^{-1})\biggr) \;.
	\end{align*}
	Defining for $u\in (0,1)$, ${\bbE[\hat{N}_{U_0,n}(u)]:= \sum_{j=1}^{\lceil k_n u\rceil}\bbE[\hat{N}_{j_0\to j,n}]}$, it follows from standard results on convergence to Poisson processes \cite{Daley:VereJones:TPP2} that $\hat{N}_{U_0,n}(u)$ converges weakly to a Poisson point process with intensity
	\begin{align} \label{eq:out:intensity}
		\overrightarrow{\lambda}(u_0,u) = \frac{\gvar_{1-\alpha}}{\mu-\alpha} \biggl(\frac{(1-\alpha)(\mu/2-1)}{(\mu-\alpha)(1-2/\gamma)} (u_0 u)^{-1/\gamma}\frac{1 - u_{+}^{1-2/\gamma}}{u_{+}^{1-2/\gamma}} + \frac{1}{2} u_{-}^{-1/\gamma}u_{+}^{-(1-1/\gamma)} \biggr) \;,
	\end{align}
	where $u_+=u_0\vee u$ and $u_-=u_0\wedge u$. Similarly, for incoming edges to $j_0$, the symmetry of the sampling process yields $\overleftarrow{\lambda}(u_0,u) = \overrightarrow{\lambda}(u_0,u)$.

	It follows that the degree of a randomly sampled vertex is a Poisson random variable with mean parameter
	\begin{align} \label{eq:poisson:mean}
		\Lambda(\gvar_{1-\alpha}, U_0) =  \int_{0}^1 \overleftarrow{\lambda}(U_0,u) + \overrightarrow{\lambda}(U_0,u) du = \gvar_{1-\alpha}(U_0^{-1/\gamma} -1) \;.
	\end{align}
	Hence, by the conjugacy relationship between the Poisson and Gamma distributions, the probability that the degree of a randomly sampled vertex, conditioned on $U_0$, is equal to $d+1$ is
	\begin{align*}
		\bbP[D=d+1 \mid U_0] & = \bbE[e^{-\Lambda(\gvar_{1-\alpha},U_0)} \Lambda(\gvar_{1-\alpha},U_0)^d \mid U_0]/d! \\
		& = \frac{\Gamma(d+1-\alpha)}{\Gamma(d+1)\Gamma(1-\alpha)} (1-U_0^{1/\gamma})^d U_0^{(1-\alpha)/\gamma} \;.
	\end{align*}
	Finally, taking the expectation with respect to the uniform random variable $U_0$,
	\begin{align*}
		\bbP[D=d+1] & = \frac{\Gamma(d+1-\alpha)}{\Gamma(d+1)\Gamma(1-\alpha)} \int_0^1 (1-u^{1/\gamma})^d u^{(1-\alpha)/\gamma} du \\
		& = \frac{\Gamma(d+1-\alpha)}{\Gamma(d+1)\Gamma(1-\alpha)} \gamma \frac{\Gamma(d+1)\Gamma(1-\alpha + \gamma)}{\Gamma(d + 1 + 1-\alpha + \gamma)} \\
		& = \gamma\frac{\Gamma(d+1-\alpha)\Gamma(1-\alpha + \gamma)}{\Gamma(d + 1 + 1-\alpha + \gamma)\Gamma(1-\alpha)} \;,
	\end{align*}
	which is the stated result. \Cref{corollary:rep:1} follows by checking moments. For $\sigma=1$, \cref{corollary:rep:2} follows from \eqref{eq:poisson:mean} by observing that $U_0^{1/\gamma}$ is distributed as $\BetaDist(\gamma,1)$; for $\sigma\in(0,1)$, it follows from a similar integral identity.
\end{proof}

\subsection{Proofs for \texorpdfstring{\cref{sec:examples}}{Section 4}}
\label{sec:proofs:examples}

\begin{proof}[Proof of \cref{prop:pa:limits}]
	Let $\bar{m}:=2m-1$ and $\bar{\alpha}:=2m-\alpha$. 
	When ${t_j = 2m(j-1) + 1}$, \eqref{eq:martingale:z} can be manipulated into the form
	\begin{align*}
		Z_n(p,t) & = \frac{\Gamma(n - k_n\alpha) \Gamma(k_n\bar{\alpha} - \bar{m} + \bar{p}) }{\Gamma(n - k_n\alpha + \bar{p}) \Gamma(k_n\bar{\alpha} - \bar{m})} \prod_{i=1}^{\bar{m}} \frac{\Gamma(k_n - i/\bar{\alpha})\Gamma(r + 1 + (\bar{p} - i)/\bar{\alpha})}{\Gamma(k_n + (\bar{p} - i)/\bar{\alpha})\Gamma(r + 1 - i/\bar{\alpha})} \dotsm \nonumber \\
		& \quad\quad \times \frac{\Gamma(r\bar{\alpha})}{\Gamma(r\bar{\alpha} + \bar{p})} \prod_{j=1}^r \frac{\Gamma(\ct_{j,n} - \alpha + p_j)}{\Gamma(\ct_{j,n} - \alpha)} \\
		& = X_n(p,t) 
			\frac{\Gamma(r\bar{\alpha})}{\Gamma(r\bar{\alpha} + \bar{p})} \prod_{i=1}^{\bar{m}} \frac{\Gamma(r+1 + \frac{\bar{p} - i}{\bar{\alpha}})}{\Gamma(r+1 - \frac{i}{\bar{\alpha}})} \;,
	\end{align*}
	where 
	\begin{align*}
		X_n(p,t) := \biggl(\frac{n}{2m}\biggr)^{-\bar{p}\frac{\bar{m}}{\bar{\alpha}}} 
			\prod_{j=1}^r \ct_{j,n}^{p_j} (1 + O(n^{-1})) \;.
	\end{align*}
	Therefore, algebraic manipulations of $Z_n(p,t)$ show that for large $n$:
	\begin{gather} 
		Z_n(p ,t) = X_n(p,t) \bbE[\prod_{i=1}^{\bar{m}} \gvar_{r+1 - \frac{i}{\bar{\alpha}}}^{\bar{p}/\bar{\alpha}}]/\bbE[\gvar_{r\bar{\alpha}}^{\bar{p}}] \label{eq:pa:id:1} \\
		Z_n(p ,t) = X_n(p,t) \bbE[\prod_{i=1}^{\bar{m}} \gvar_{1 - \frac{i}{\bar{\alpha}}}^{\bar{p}/\bar{\alpha}}] / \bbE[\gvar_{r\bar{\alpha}}^{\bar{p}} \prod_{j=1}^{r} \bvar^{\bar{p}}_{j\bar{\alpha}-\bar{m},\bar{m}}] \label{eq:pa:id:2} \\
		Z_n(p ,t) = X_n(p,t) \bbE[\prod_{i=1}^{\bar{m}} \gvar_{1 - \frac{i}{\bar{\alpha}}}^{\bar{p}/\bar{\alpha}}] / \bbE[\gvar_{r\bar{\alpha}-\bar{m}}^{\bar{p}} \prod_{j=1}^{r-1} \bvar^{\bar{p}}_{j\bar{\alpha}-\bar{m},\bar{m}}] \label{eq:pa:id:3} \\
		Z_n(p ,t) = X_n(p,t) \bbE[\prod_{i=1}^{\bar{m}} \gvar_{1 - \frac{i}{\bar{\alpha}}}^{\bar{p}/\bar{\alpha}}] \bbE[\prod_{j=2}^r \bvar^{\bar{p}}_{(j-1)\bar{\alpha},1-\alpha}] / \bbE[\gvar_{1-\alpha}^{\bar{p}}] \label{eq:pa:id:4} \;.
	\end{gather}
	On the other hand, \cref{lemma:degree:sequence:martingale} establishes that $Z_n(p,t)$ converges almost surely to the product of random variables $\xi_1^{p_1}\dotsm\xi_r^{p_r}$. Now, letting $p\setminus p_j$ denote $p$ with the $j$th element set to zero,
	\begin{align*}
		& \bbE[\xi_1^{p_1}\dotsm\xi_r^{p_r}] = \bbE[\lim_{n\to\infty} Z_n(p,t)]  = \bbE[Z_{t_r}(p,t)] \\
		& = \bbE[Z_{t_r-1}(p\setminus p_r,t)] \frac{\Gamma(1 - \alpha + \bar{p}) \Gamma(r\bar{\alpha} - \bar{m}) \Gamma((r-1)\bar{\alpha} + \bar{p}_{r-1})}{\Gamma(1-\alpha)\Gamma(r\bar{\alpha} - \bar{m} + \bar{p}_r) \Gamma((r-1)\bar{\alpha}) }  \\
		& =\prod_{j=2}^r \frac{\Gamma(1 - \alpha + p_j) \Gamma(j\bar{\alpha} - \bar{m}) \Gamma((j-1)\bar{\alpha} + \bar{p}_{j-1})}{\Gamma(1 - \alpha) \Gamma((j-1)\bar{\alpha}) \Gamma(j\bar{\alpha} - \bar{m} + \bar{p}_j)} = \bbE[\prod_{j=2}^r \Psi_j^{p_j} (1-\Psi_j)^{\bar{p}_{j-1}}] \;.
	\end{align*}
	By equating \eqref{eq:pa:id:1}--\eqref{eq:pa:id:4} to $\bbE[Z_{t_r}(p,t)]$, the results follow.
\end{proof}

\begin{proof}[Proof of \cref{prop:crp:limits}]
	Define $T_{1:r}:=(T_1,T_2,\dotsc,T_r)$ and
	\begin{align} \label{eq:crp:martingale}
		Z_n^{\alpha,\theta}(p,T_{1:r}) &:= \frac{\Gamma(n + \theta)}{\Gamma(n + \theta + \bar{p})} \prod_{j=1}^r \frac{\Gamma(\ct_{j,n}-\alpha + p_j)}{\Gamma(\ct_{j,n} - \alpha)} \\
		& = n^{-\bar{p}} \prod_{j=1}^r \ct_{j,n}^{p_j} (1 + O(n^{-1})) \;.
	\end{align}
	The fact that ${\bbE[Z_{n+1}^{\alpha,\theta}(p,T_{1:r}) \mid \mathcal{F}_n] = Z_n^{\alpha,\theta}(p,T_{1:r})}$ shows that $Z_n^{\alpha,\theta}(p,T_{1:r})$ is a nonnegative martingale for $n\geq T_r$. $Z_n^{\alpha,\theta}(p,T_{1:r})$ can be bounded in $L_2$ following an argument similar to that given in the proof of \cref{lemma:degree:sequence:martingale}, and hence it converges almost surely to $\xi_1^{p_1}\dotsm \xi_r^{p_r}$. Therefore,
	\begin{align*}
		& \bbE[\lim_{n\to\infty} n^{-\bar{p}} \prod_{j=1}^r \ct_{j,n}^{p_j} \mid T_{1:r}] = \bbE[\xi_1^{p_1}\dotsm \xi_r^{p_r} \mid T_{1:r}] \\
		& = \frac{\Gamma(1 - \alpha + p_r)\Gamma(T_r + \theta)\Gamma(T_r - 1- (r-1)\alpha + \bar{p}_{r-1})}{\Gamma(1 - \alpha)\Gamma(T_r + \theta + p_1)\Gamma(T_r - 1 - (r-1)\alpha)} \dots \\
		& \quad \times \prod_{j=2}^{r-1} \frac{\Gamma(1 - \alpha + p_j)\Gamma(T_j - j\alpha)\Gamma(T_j - 1 - (j-1)\alpha + \bar{p}_{j-1})}{\Gamma(1-\alpha)\Gamma(T_j - j\alpha + \bar{p}_j)\Gamma(T_j - 1 - (j-1)\alpha )} \\
		& = \bbE[\bvar_{T_r-r\alpha,\theta+r\alpha}^{\bar{p}} \cdot \prod_{j=2}^r \Psi_j^{p_j} (1-\Psi_j)^{\bar{p}_{j-1}} ] \;,
	\end{align*}
	from which the result follows.
\end{proof}

\begin{proof}[Proof of \cref{prop:crp:marginal}]
	For an exchangeable Gibbs partition with $\alpha\in(0,1)$ and a sequence of coefficients $v_{n,k}$, one may construct a martingale similar to that in \eqref{eq:crp:martingale}. It is then straightforward to show that the scaled degrees converge almost surely to random variables $(\xi_j)_{1\leq j\leq r}$, conditionally on $T_1,\dotsc,T_r$. Conditioned on $T_j$, $\xi_j$ has marginal moments
	\begin{align*}
		\bbE[\xi_j^p \mid T_j] = \frac{\Gamma(1 - \alpha + p)}{\Gamma(1-\alpha)} \bigl( 1 + p\frac{V_{T_j+1,j}}{V_{T_j,j}} \bigr)^{-1} \;.
	\end{align*}
	\citet{Kerov:2006} showed that the only exchangeable Gibbs partitions with coefficients that can be represented as a ratio $V_{n,k}=v_k/c_n$ are those of the $\CRP(\alpha,\theta)$, in which case the product in the previous equation becomes independent of $j$. 
	See also \cite[Lemma 4.1]{Griffiths:Spano:2007}, which also implies the result.
\end{proof}

\begin{proof}[Proof of \cref{prop:ys:limit}]
	Define for $p > -1$,
	\begin{align*}
		Z_n^{\beta}(p,T_{1:r}) & := \frac{\Gamma(n)}{\Gamma(n + \bar{p}(1-\beta))} \prod_{j=1}^r \frac{\Gamma(\ct_{j,n} + p_j)}{\Gamma(\ct_{j,n})} \\
		& = n^{-\bar{p}(1-\beta)} \prod_{j=1}^r \ct_{j,n}^{p_j} (1 + O(n^{-1})) \;,
	\end{align*}
	which is by construction a nonnegative martingale, bounded in $L_2$, for $n\geq T_r$. Therefore,
	\begin{align*}
		& \bbE[\lim_{n\to\infty} n^{-\bar{p}(1-\beta)} \prod_{j=1}^r \ct_{j,n}^{p_j}] = \bbE[\xi_1^{p_1} \dotsm \xi_r^{p_r}] \\
		& = \frac{\Gamma(T_r + \bar{p})}{\Gamma(T_r + \bar{p}(1-\beta))} \prod_{j=2}^r \frac{\Gamma(1 + p_j)\Gamma(T_j) \Gamma(T_j - 1 + \bar{p}_{j-1})}{\Gamma(T_j - 1)\Gamma(T_j + \bar{p}_j)} \\
		& = \bbE[\mlvar_{1-\beta,T_r-1}^{-\bar{p}(1-\beta)}] \bbE[\bvar^{\bar{p}}_{T_r,(T_r-1)\frac{\beta}{1-\beta}}] \bbE[\prod_{j=2}^r \Psi_j^{p_j} (1-\Psi_j)^{\bar{p}_{j-1}} ] \;,
	\end{align*}
	from which the result follows. The marginal identities result from altering the above martingale to contain only $D_{j,n}$, which is also a martingale. Similar treatment of the moments, along with the beta-gamma algebra and the identity (1.3) in \citet{James:2015aa} yields the rest of the identities.
\end{proof}

\begin{proof}[Proof of \cref{prop:immigration:urn}]
	In analogy to the previous proof, for every $p>-1$,
	\begin{align*}
		Z_n^{\beta}(p,w+b) & = \frac{\Gamma(n)}{\Gamma(n + p(1-\beta))} \frac{\Gamma(\ct_{w,n} + p)}{\Gamma(\ct_{w,n})} \\
		& = n^{-p(1-\beta)} \ct_{w,n}^{p} (1 + O(n^{-1}))
	\end{align*}
	is a nonnegative martingale, bounded in $L_2$, for $n\geq w+b$. Therefore, it converges almost surely to a random variable $\xi_{w,w+b}^p$, where
	\begin{align*}
		\bbE[\xi_{w,w+b}^p] & = \frac{\Gamma(w+b) \Gamma(w + p)}{\Gamma(w+b+p(1-\beta)) \Gamma(w)} \\
		& = \bbE[\gvar_w^p] / \bbE[\gvar_{w+b}^{p(1-\beta)}] = \bbE[\bvar_{w,b}^p \bvar_{w+b,(w+b-1)\frac{\beta}{1-\beta}}^p \mlvar_{1-\beta,w+b-1}^p] \;.
	\end{align*}
	The other identities can be verified by checking moments; they also follow from the beta-gamma algebra and from \cite{James:2015aa}.
\end{proof}

\end{appendices}

\end{document}